\documentclass[11pt]{article}

\usepackage{amsthm}
\usepackage{amsmath}
\usepackage{amssymb}
\usepackage{amsmath, amsthm, bbm}
\usepackage{framed, fullpage}

\let\Lpolish\L

\usepackage{mathrsfs}

\usepackage[backref,colorlinks,citecolor=blue,bookmarks=true]{hyperref}

\newtheorem*{theo*}{Theorem}
\newtheorem{theo}{Theorem}

\newtheorem{lemma}{Lemma}[section]
\newtheorem{coro}[theo]{Corollary}
\newtheorem{definition}[lemma]{Definition}
\newtheorem{prop}[lemma]{Proposition}

\newtheorem{claim}[lemma]{Claim}
\newtheorem{facts}{Fact}
\newtheorem{fact}[facts]{Fact}

\newtheorem{remark}[lemma]{Remark}

\makeatletter
\renewenvironment{proof}[1][\proofname]
{\par\pushQED{\qed}
	\normalfont\topsep6\p@\@plus6\p@\relax\trivlist
	\item[\hskip\labelsep\bfseries#1\@addpunct{.}]
	\ignorespaces}
{\popQED\endtrivlist\@endpefalse}
\makeatother

\newcommand{\F}{\mathcal{F}}

\newcommand{\N}{\mathbb{N}}

\newcommand{\Ex}{\mathbb{E}}

\newcommand{\ith}{i\text{-th}}

\newcommand{\norm}[1]{\left\lVert #1 \right\rVert}

\renewcommand{\P}{\mathcal{P}}
\renewcommand{\Pr}{\mathbb{P}}
\newcommand{\Q}{\mathcal{Q}}
\newcommand{\R}{\mathcal{R}}
\newcommand{\E}{\mathcal{E}}
\newcommand{\A}{\mathcal{A}}

\newcommand{\V}{\mathcal{V}}

\DeclareMathOperator{\poly}{poly}

\renewcommand{\a}{\alpha}
\renewcommand{\b}{\beta}
\renewcommand{\d}{\delta}
\newcommand{\g}{\gamma}

\renewcommand{\L}{\mathcal{L}}
\newcommand{\sub}{\subseteq}
\newcommand{\sm}{\setminus}

\newcommand{\e}{\epsilon}

\newcommand{\Z}{\mathcal{Z}}

\newcommand{\HH}{\mathcal{H}}

\newcommand{\X}{\mathcal{X}}
\newcommand{\Y}{\mathcal{Y}}

\newcommand{\D}{\mathcal{D}}

\newcommand{\K}{\mathcal{K}}
\newcommand{\G}{\mathcal{G}}

\newcommand{\h}[1]{\widehat{#1}}

\newcommand{\Lside}{\mathbf{L}}
\newcommand{\Rside}{\mathbf{R}}

\newcommand{\Hy}[1]{H}

\DeclareMathOperator{\Exp}{\mathbb{E}}
\DeclareMathOperator{\codeg}{codeg}
\DeclareMathOperator{\Var}{Var}

\DeclareMathOperator{\Ack}{Ack}

\DeclareMathOperator{\Cross}{Cross}

\renewcommand{\k}{\kappa}

\newcommand{\Vside}{\mathbf{V}}

\DeclareMathOperator{\under}{U}

\date{}

\title{A Tight Bound for Hyperaph Regularity}
\author{Guy Moshkovitz\thanks{School of Mathematical Sciences, Tel Aviv University, Tel Aviv 6997801, Israel. Email: {\tt guymoshkov@gmail.com}. Supported in part by ERC Starting Grant 633509.} 
\and Asaf Shapira\thanks{School of Mathematical Sciences, Tel Aviv University, Tel Aviv 6997801, Israel. Email: {\tt asafico@tauex.tau.ac.il}. Supported in part by ISF Grant 1028/16 and ERC Starting Grant 633509.}}
\begin{document}

\maketitle
\begin{abstract}
The hypergraph regularity lemma -- the extension of Szemer\'edi's graph regularity lemma to the setting of $k$-uniform hypergraphs -- is one of the most celebrated combinatorial results obtained in the past decade.
By now there are several (very different) proofs of this lemma, obtained by Gowers, by Nagle-R\"odl-Schacht-Skokan and by Tao.
Unfortunately, what all these proofs have in common is that they yield regular partitions whose order is given by the $k$-th Ackermann function.

We show that such Ackermann-type bounds are unavoidable for every $k \ge 2$, thus confirming a prediction of Tao.
Prior to our work, the only result of this type was Gowers' famous lower bound for graph regularity.

\end{abstract}

\section{Introduction}

As part of the proof of his eponymous theorem~\cite{Szemeredi75} on arithmetic progressions in dense sets of integers, Szemer\'edi developed (a variant of what is now known as) the graph {\em regularity lemma}~\cite{Szemeredi78}.
The lemma roughly states that the vertex set of every graph can be partitioned into a bounded number of parts such that
almost all the bipartite graphs induced by pairs of parts in the partition are quasi-random.
In the past four decades this lemma has become one of the (if not the) most powerful tools in extremal combinatorics, with applications in many other areas of mathematics. We refer the reader to~\cite{KomlosShSiSz02,RodlSc10}
for more background on the graph regularity lemma, its many variants and its numerous applications.

Perhaps the most important and well-known application of the graph regularity lemma is the original
proof of the {\em triangle removal lemma}, which states that if an $n$-vertex graph $G$ contains only $o(n^3)$ triangles, then one can turn $G$ into a triangle-free graph by removing only $o(n^2)$ edges (see \cite{ConlonFox13} for more details).
It was famously observed by Ruzsa and Szemer\'edi~\cite{RuzsaSz76} that the triangle removal lemma implies Roth's theorem~\cite{Roth54}, the special case of Szemer\'edi's theorem for $3$-term arithmetic progressions.
The problem of extending the triangle removal lemma to the hypergraph
setting was raised by Erd\H{o}s, Frankl and R\"odl~\cite{ErdosFrRo86}. One of the main motivations for obtaining such a result was the observation of Frankl and R\"odl~\cite{FrankRo02} (see also~\cite{Solymosi04}) that such a result would allow one to extend the Ruzsa--Szemer\'edi~\cite{RuzsaSz76} argument and thus obtain an alternative proof of Szemer\'edi's theorem for progressions of arbitrary length.

The quest for a hypergraph regularity lemma, which would allow one to prove a hypergraph removal lemma, took about 20 years.
The first milestone was the result of Frankl and R\"odl~\cite{FrankRo02}, who obtained a regularity lemma for $3$-uniform hypergraphs. About 10 years later, the approach of~\cite{FrankRo02} was extended to hypergraphs of arbitrary uniformity by R\"odl, Skokan, Nagle and Schacht~\cite{NagleRoSc06, RodlSk04}.
At the same time, Gowers~\cite{Gowers07} obtained an alternative version of the regularity lemma for $k$-uniform hypergraphs (from now on we will use $k$-graphs instead of $k$-uniform hypergraphs).
Shortly after, Tao~\cite{Tao06} and R\"odl and Schacht~\cite{RodlSc07,RodlSc07-B} obtained two more versions of the lemma.

As it turned out, the main difficulty with obtaining a regularity lemma for $k$-graphs was defining the correct notion of hypergraph regularity that would:
$(i)$ be strong enough to allow one to prove a counting lemma, and
$(ii)$ be weak enough to be satisfied by every hypergraph (see the discussion in~\cite{Gowers06} for more on this issue).
And indeed, the above-mentioned variants of the hypergraph regularity lemma rely on four different notions of quasi-randomness, which
to this date are still not known to be equivalent\footnote{This should be contrasted with the setting of graphs in which (almost) all notions of quasi-randomness are not only known to be equivalent but even effectively equivalent. See e.g.~\cite{ChungGrWi89}.} (see~\cite{NaglePoRoSc09} for some partial results). What all of these proofs {\em do} have in common however,
is that they supply only Ackermann-type bounds for the size of a regular partition.\footnote{Another variant of the hypergraph regularity lemma was obtained in~\cite{ElekSz12}. This approach does not supply any quantitative bounds.}
More precisely, if we let $\Ack_1(x)=2^x$ and then define $\Ack_k(x)$ to be the $x$-times iterated\footnote{$\Ack_2(x)$ is thus a tower of exponents of height $x$, 
$\Ack_3(x)$ is the so-called wowzer function, etc.} version of $\Ack_{k-1}$, then all the above proofs guarantee to produce a regular partition of a $k$-graph whose order can be bounded from above by an $\Ack_k$-type function.

One of the most important applications of the $k$-graph regularity lemma was that it gave the first explicit
bounds for the multidimensional generalization of Szemer\'edi's theorem, see~\cite{Gowers07}. The original proof of this result,
obtained by Furstenberg and Katznelson~\cite{FurstenbergKa78}, relied on Ergodic Theory and thus supplied no quantitative bounds at all.
Examining the reduction between these theorems~\cite{Solymosi04} reveals that if one could improve the Ackermann-type bounds for the $k$-graph regularity
lemma, by obtaining (say) $\Ack_{k_0}$-type upper bounds (for all $k$), then one would obtain the first primitive recursive bounds for the
multidimensional generalization of Szemer\'edi's theorem. Let us note that obtaining such bounds just for
van der Waerden's theorem~\cite{Shelah89} and Szemer\'edi's theorem~\cite{Szemeredi75} (which are two special case) were open problems for many decades till
they were finally solved by Shelah~\cite{Shelah89} and Gowers~\cite{Gowers01}, respectively.
Further applications of the $k$-graph regularity lemma (and the hypergraph removal lemma in particular) are described in~\cite{RodlNaSkScKo05} and~\cite{RodlTeScTo06} as well as in R\"odl's recent ICM survey~\cite{Rodl14}.

A famous result of Gowers~\cite{Gowers97} states that the $\Ack_2$-type upper bounds for graph regularity
are unavoidable. Several improvements~\cite{FoxLo17},
variants~\cite{ConlonFo12,KaliSh13,MoshkovitzSh18} and simplifications~\cite{MoshkovitzSh16} of Gowers' lower bound were recently obtained, but no analogous lower bound was derived even for $3$-graph regularity.
The numerous applications of the hypergraph regularity lemma naturally lead to the question of whether one can improve upon the Ackermann-type
bounds mentioned above and obtain primitive recursive bounds
for the $k$-graph regularity lemma. Tao~\cite{Tao06-h} predicted that the answer to this question is negative, in the sense that one cannot obtain better than $\Ack_k$-type upper bounds for the $k$-graph regularity lemma for every $k \ge 2$. The main result presented in this paper confirms this prediction.

\begin{theo}{\bf[Main result, informal statement]}\label{thm:main-informal}
	The following holds for every $k\geq 2$: every regularity lemma for $k$-graphs satisfying some
	mild conditions can only guarantee to produce partitions of size bounded by an $\Ack_k$-type function.
\end{theo}

Our main result, stated formally as Theorem~\ref{theo:main} in Subsection \ref{subsec:formal}, 
establishes an $\Ack_k$-type lower bound for \emph{$\langle \d \rangle$-regularity} of $k$-graphs, which is a new notion 
of graph/hypergraph regularity which we introduce in this paper.

We will demonstrate the effectiveness of Theorem \ref{thm:main-informal} by deriving from it a lower bound for
the $k$-graph regularity lemma of R\"odl and Schacht~\cite{RodlSc07}.

\begin{coro}[Lower bound for $k$-graph regularity]\label{coro:RS-LB}
	For every $k\geq 2$, there is an $\Ack_k$-type lower bound for the $k$-graph regularity lemma of R\"odl and Schacht~\cite{RodlSc07}.
\end{coro}

As we discuss at the beginning of Section~\ref{sec:coro}, the lower bound stated in Corollary~\ref{coro:RS-LB} holds even for
a very weak/special case of the $k$-graph regularity lemma of~\cite{RodlSc07}. We also note that when specialized to $k=3$,
the notion of regularity used by R\"odl and Schacht~\cite{RodlSc07} is at least as strong as those used by
Frankl and R\"odl~\cite{FrankRo02} and by Gowers~\cite{Gowers06}. Hence as a special case of Corollary~\ref{coro:RS-LB} we also obtain\footnote{Since the full proofs of these assertions are given in a companion manuscript we put on the Arxiv~\cite{MS3}, we will not
explain here why the $3$-graph regularity notion of \cite{RodlSc07} is indeed stronger than those used in \cite{FrankRo02,Gowers06}.} tight $\Ack_3$-type lower bounds for the $3$-graph regularity lemmas obtained in \cite{FrankRo02,Gowers06}.

\subsection{Barriers and other aspects of the main proof}

Before getting into the gory details of the proof, let us informally discuss what we think are some
interesting aspects of the proof of Theorem~\ref{thm:main-informal}.

\paragraph{Why is it hard to ``step up''?}
The reason why the upper bound for graph regularity is of tower-type
is that the process of constructing a regular partition of a graph proceeds by a sequence of steps, each increasing the size of the partition exponentially.
The main idea behind Gowers' lower bound for graph regularity~\cite{Gowers97} is in ``reverse engineering'' the proof of the upper bound; in other words,
in showing that (in some sense) the process of building the partition using a sequence of exponential refinements is unavoidable.
Now, a common theme in all proofs of the hypergraph regularity lemma
for $k$-graphs is that they proceed by induction on $k$; that is, in the process of constructing a regular
partition of the input $k$-graph $H$, the proof applies the $(k-1)$-graph regularity lemma on certain $(k-1)$-graphs
derived from $H$. This is why one gets $\Ack_k$-type upper bounds. So with~\cite{Gowers97} in mind, one might guess
that in order to prove a matching lower bound one should ``reverse engineer'' the proof of the upper bound and show that such a process is unavoidable. However, this turns out to be false! As we argued in~\cite{MoshkovitzSh18}, in order to prove an {\em upper bound} for (say) $3$-graph regularity it is in fact enough to iterate a relaxed version of graph regularity which we call the ``sparse regular approximation lemma'' (SRAL for short).
Therefore, in order to prove an $\Ack_3$-type {\em lower bound} for $3$-graph
regularity one cannot simply ``step up'' an $\Ack_2$-type lower bound for graph regularity. Indeed, a necessary condition
would be to prove an $\Ack_2$-type lower bound for SRAL. See also the discussion following Lemma~\ref{theo:core} in Subsection~\ref{subsec:overview}
on how do we actually use a graph construction in order to get a 
hypergraph construction.

\paragraph{A new notion of graph/hypergraph regularity:}
In a recent paper \cite{MoshkovitzSh18} we proved an $\Ack_2$-type
lower bound for SRAL.
As it turned out, even this lower bound was not enough to allow us to step up the graph lower bound
into a $3$-graph lower bound. To remedy this, in the present paper we introduce an even weaker notion of graph/hypergraph regularity
which we call $\langle \d \rangle$-regularity. 
(Henceforth, we think of $\d$ as a fixed constant.)
This notion seems to be right at the correct level of ``strength'';
on the one hand, it is strong enough to allow one to prove $\Ack_{k-1}$-type lower bounds for $(k-1)$-graph regularity, while at the
same time weak enough to allow one to induct, that is, to use it in order to then prove $\Ack_{k}$-type lower bounds for $k$-graph regularity.
Another critical feature of our new notion of hypergraph regularity is that it has (almost) nothing to do with hypergraphs!
A disconcerting aspect of all proofs of the hypergraph regularity lemma is that they involve a very complicated nested/inductive structure.
Furthermore, one has to introduce an elaborate hierarchy of constants that controls how regular one level of the partition is compared to
the previous one. What is thus nice about our new notion is that it involves only various kinds of instances of graph $\langle \d \rangle$-regularity.
As a result, our proof is (relatively!) simple.

\paragraph{How do we find witnesses for $k$-graph irregularity?}
The key idea in Gowers' lower bound~\cite{Gowers97} for graph regularity was in constructing a graph $G$, based on a sequence of
partitions ${\cal P}_1,{\cal P}_2,\ldots$ of $V(G)$, with the following
inductive property: if a vertex partition $\Z$ refines ${\cal P}_i$
but does not refine ${\cal P}_{i+1}$ then $\Z$ is not $\epsilon$-regular.
The key step of the proof of~\cite{Gowers97} is in finding witnesses showing
that pairs of clusters of $\Z$ are irregular. The main difficulty in extending this
strategy to $k$-graphs already reveals itself in the setting of $3$-graphs. In a nutshell, while
in graphs, a witness to irregularity of a pair of clusters $A,B \in \Z$ is
{\em any} pair of large subsets $A' \sub A$ and $B' \sub B$, in the setting of $3$-graphs
we have to find three large edge-sets
(usually referred to as a {\em triad} in the hypergraph regularity literature) that
have an additional property: they must together form a graph containing many triangles.
It thus seems quite hard to extend Gowers' approach already to the setting of $3$-graphs. By working
with the much weaker notion of $\langle \d \rangle$-regularity, we circumvent this
issue since two of the edges sets in our version of a triad are always complete bipartite graphs.

\paragraph{What is then the meaning of Theorem \ref{thm:main-informal}?}
Our main result, stated formally as Theorem~\ref{theo:main}, establishes an $\Ack_k$-type lower bound
for $\langle \d \rangle$-regularity of $k$-graphs, that is, for a specific new
version of the hypergraph regularity lemma. Therefore, we immediately get $\Ack_k$-type lower bounds for
any $k$-graph regularity lemma which is at least as strong as our new lemma, that is, for any lemma
whose requirements/guarantees imply those that are needed in order to satisfy our new notion of regularity.
In particular, we will prove Corollary \ref{coro:RS-LB} by showing that the $k$-graph regularity notion
used in \cite{RodlSc07} is at least as strong as $\langle \d \rangle$-regularity.

\paragraph{How strong is our lower bound?} 
Since Theorem \ref{thm:main-informal} gives a lower bound for
$\langle \d \rangle$-regularity and Corollary \ref{coro:RS-LB} shows that
this notion is at least as weak as previously used notions of regularity, it is natural to ask:
$(i)$ is this notion equivalent to one of the other notions? $(ii)$ is this notion strong enough for proving the
hypergraph removal lemma, which was one of the main reasons for developing
the hypergraph regularity lemma? We will prove that the answer to both questions is {\em negative} by showing that already for
graphs, $\langle \d \rangle$-regularity (for $\d$ a fixed constant) is not strong enough even for proving the triangle removal lemma.
This of course makes our lower bound even stronger as it already applies to a very weak notion of regularity.
This is formally stated as Proposition \ref{claim:example} in Section \ref{sec:example}. 
See also the discussion following Theorem~\ref{theo:main} and the one at the beginning of Section~\ref{sec:example}.

\paragraph{How tight is our bound?} Roughly speaking, we will show that for a $k$-graph with $pn^k$ edges,
every $\langle \d \rangle$-regular partition has order at least $\Ack_k(\log 1/p)$. In a recent paper \cite{MoshkovitzSh16}
we proved that in graphs, one can prove a matching $\Ack_2(\log 1/p)$ upper bound, even for a slightly stronger notion than $\langle \d \rangle$-regularity.
This allowed us to obtain a new proof of Fox's $\Ack_2(\log 1/\epsilon)$ upper bound for the graph removal lemma \cite{Fox11} (since the stronger notion allows to count small subgraphs). We believe that it should
be possible to match our lower bounds with $\Ack_k(\log 1/p)$ upper bounds (even for a slightly stronger notion analogous to the one used in \cite{MoshkovitzSh16}). We think that it should be possible to deduce from such an upper bound an $\Ack_k(\log 1/\epsilon)$ upper bound
for the $k$-graph removal lemma. The best known bounds for this problem are (at least) $\Ack_k(\poly(1/\epsilon))$.

\paragraph{How is this paper related to~\cite{MS3}?}
For the reader's convenience we have put on the Arxiv a companion manuscript \cite{MS3} in which we give
a completely self contained proof of the special case of Theorem \ref{theo:main} for $3$-graphs. 
First, the definitions given in Section~\ref{sec:define} when specialized to $k=3$ are the same notions used in (Section~2 of) \cite{MS3}.
The heart of the proof of Theorem~\ref{theo:main} is given by Lemma~\ref{lemma:ind-k} which is proved by induction on $k$.
The (base) case $k=2$ follows easily from Lemma~\ref{theo:core}.
Hence, the heart of the matter is the proof of Lemma~\ref{lemma:ind-k} by induction on $k$. Within this framework,
the argument given in~\cite{MS3} is precisely the deduction of Lemma~\ref{lemma:ind-k} for $k=3$ from the case $k=2$.
Hence, the reader interested in seeing the inductive proof of Lemma~\ref{lemma:ind-k} ``in action''---without the clutter
caused by the complicated definitions related to $k$-graphs---is advised to check~\cite{MS3}.
We also mention that in \cite{MS3} we deduce from the $3$-graph lower bound a tight bound for the regularity lemmas
of Frankl and R\"odl~\cite{FrankRo02} and of Gowers~\cite{Gowers07}. These proofs can also be considered as special cases of 
the arguments we give here in order to prove Corollary \ref{coro:RS-LB}.

\subsection{Paper overview}

Broadly speaking, Section~\ref{sec:define} serves as the technical introduction to this paper, while Sections~\ref{sec:LB} and~\ref{sec:core} contain the main technical proofs.
More concretely, in Section~\ref{sec:define} we will first define the new notion of $k$-graph regularity, which we term $\langle \d \rangle$-regularity,
for which we will prove our main lower bound. We will then give the formal version of Theorem~\ref{thm:main-informal} (see Theorem~\ref{theo:main}).
This will be followed by the statement of the main technical result we will use in this paper, Lemma~\ref{theo:core},
and an overview of how this technical result is used in the proof of Theorem~\ref{theo:main}.
The proof of Theorem~\ref{theo:main} appears in Section~\ref{sec:LB} 
and the proof of Lemma~\ref{theo:core} appears
in Section~\ref{sec:core}. To make it easier to read this paper
we made sure that these two sections are completely independent of each other in the sense that the only result
of Section~\ref{sec:core} that is used in the proof of Theorem~\ref{theo:main} is Lemma~\ref{theo:core}.
In Section \ref{sec:coro} we describe how Theorem~\ref{theo:main} can be used in order to prove Corollary~\ref{coro:RS-LB}, thus establishing tight
$\Ack_k$-type lower bounds for a concrete version of the hypergraph regularity lemma.
Finally, in Section~\ref{sec:example} we prove Proposition~\ref{claim:example} by describing an example showing that even in the setting of graphs,
$\langle \d \rangle$-regularity is strictly weaker than the usual notion of graph regularity, as it does not allow one even to count triangles.

\section{$\langle \d \rangle$-regularity and Proof Overview}\label{sec:define}

\subsection{Preliminary definitions}\label{subsec:preliminaries}

%

Before giving the definition of $\langle \d \rangle$-regularity, let us start with some standard definitions regarding partitions of hypergraphs.
Formally, a \emph{$k$-graph} is a pair $H=(V,E)$, where $V=V(H)$ is the vertex set and $E=E(H) \sub \binom{V}{k}$ is the edge set of $H$.
The number of edges of $H$ is denoted $e(H)$ (i.e., $e(H)=|E|$).
We denote by $K^k_\ell$ the complete $\ell$-vertex $k$-graph (i.e., containing  all possible $\binom{\ell}{k}$ edges).
The $k$-graph $H$ is \emph{$\ell$-partite} $(\ell \ge k)$ on (disjoint) vertex classes $(V_1,\ldots,V_\ell)$ if every edge of $H$ has at most one vertex from each $V_i$.
We denote by $H[V'_1,\ldots,V'_\ell]$ the $\ell$-partite $k$-graph induced on vertex subsets $V_1' \sub V_1,\ldots,V_\ell'\sub V_\ell$; that is, $H[V'_1,\ldots,V'_\ell]=((V_1',\ldots,V_\ell'),\, \{e \in E(H) \,\vert\, \forall i \colon e \cap V_i \in V_i'\})$.
The \emph{density} $d(H)$ of a $k$-partite $k$-graph $H$ is $e(H)/\prod_{i=1}^k |V_i|$.
The set of edges of $G$ between disjoint vertex subsets $A$ and $B$ is denoted by $E_G(A,B)$; the density of $G$ between $A$ and $B$ is denoted by $d_G(A,B)=e_G(A,B)/|A||B|$, where $e_G(A,B)=|E_G(A,B)|$. We use $d(A,B)$ if $G$ is clear from context.
When it is clear from context, we sometimes identify a hypergraph with its edge set. In particular, we will write $V_1 \times V_2$ for the complete bipartite graph on vertex classes $(V_1,V_2)$, and more generally, $V_1 \times\cdots\times V_k$ for the complete $k$-partite $k$-graph on vertex classes $(V_1,\ldots,V_k)$.

For a partition $\Z$ of a vertex set $V$,
the complete multipartite $k$-graph on $\Z$ is denoted by
$\Cross_k(\Z)=
\big\{ e \sub V \,\big\vert\, \forall\, W \in \Z \colon |e \cap W| \le 1 \text{ and } |e|=k\big\}$.
For partitions $\P,\Q$ of the same underlying set, we say that $\Q$ \emph{refines} $\P$, denoted $\Q \prec \P$, if every member of $\Q$ is contained in a member of $\P$.
We say that $\P$ is \emph{equitable} if all its members have the same size.\footnote{In a regularity lemma one allows the parts to differ in size by at most $1$ so that it applies to all (hyper-)graphs. For our lower bound this is unnecessary.}
We use the notation $x \pm \e$ for a number lying in the interval $[x-\e,\,x+\e]$.

We now define a \emph{$k$-partition},
which is a notion of a hypergraph partition\footnote{This is a standard notion, identical to the one used by R\"odl and Schacht (\cite{RodlSc07}, Definition10).}.
A $k$-partition $\P$ is of the form $\P=\P^{(1)} \cup\cdots\cup \P^{(k)}$ where $\P^{(1)}$ is a vertex partition, and for each $2 \le r \le k$, $\P^{(r)}$ is a partition of $\Cross_r(\P^{(1)})$ satisfying a condition we will state below.
First, to ease the reader in, let us describe here what a $k$-partition is for $1 \le k \le 3$.
A $1$-partition is simply a vertex partition.
A $2$-partition $\P=\P^{(1)} \cup \P^{(2)}$ consists of a vertex partition $\P^{(1)}$ and a partition $\P^{(2)}$ of $\Cross_2(\P^{(1)})$
such that the complete bipartite graph between any two distinct clusters of $\P^{(1)}$ is a union of parts of $\P^{(2)}$.
A $3$-partition $\P=\P^{(1)} \cup \P^{(2)} \cup \P^{(3)}$ consists of a $2$-partition $\P^{(1)} \cup \P^{(2)}$ and a partition $\P^{(3)}$ of $\Cross_3(\P^{(1)})$
such that for every tripartite graph $G$ whose three vertex clusters lie in $\P^{(1)}$ and three bipartite graphs lie in $\P^{(2)}$, the $3$-partite $3$-graph consisting of all triangles in $G$ is a union of parts of $\P^{(3)}$.


Before defining a $k$-partition in general, we need some terminology.
A \emph{$k$-polyad} is simply a $k$-partite $(k-1)$-graph.
Thus, a $2$-polyad is just a pair of disjoint vertex sets, and a $3$-polyad is a tripartite graph.
In the rest of this paragraph let $P$ be a $k$-polyad on vertex classes $(V_1,\ldots,V_k)$.
We
often identify $P$
with the $k$-tuple $(F_1,\ldots,F_k)$ where each $F_i$ is the induced $(k-1)$-partite $(k-1)$-graph $F_i=P[\bigcup_{j \neq i} V_j]$.
%
We denote by $\K(P)$ the set of $k$-element subsets of $V(P)$ that span a clique (i.e., a $K^{k-1}_{k}$) in $P$; we view $\K(P)$ as a $k$-graph on $V(P)$.
Note that $\K(P)$ is a $k$-partite $k$-graph.
For example, if $P$ is a $2$-polyad then $\K(P)$ is a complete bipartite graph (since $K^1_2$ is just a pair of vertices), and if $P$ is a $3$-polyad then $\K(P)$ is the $3$-partite $3$-graph whose edges correspond to the triangles in $P$.
For a family of hypergraphs $\P$, we say that the $k$-polyad $P=(F_1,\ldots,F_k)$ is a $k$-polyad \emph{of} $\P$ if $F_i \in \P$ for every $1 \le i \le k$.



%


%
We are now ready to define a $k$-partition for arbitrary $k$.

\begin{definition}[$k$-partition]\label{def:r-partition}
	$\P$ is a \emph{$k$-partition} $(k \ge 1)$ on $V$ if $\P=\P^{(1)} \cup\cdots\cup \P^{(k)}$ with $\P^{(1)}$ a partition of $V$,
	and for every $2 \le r \le k$, $\P^{(r)}$ is a partition of $\Cross_r(\P^{(1)})$ into $r$-partite $r$-graphs with $\P^{(r)} \prec \K_r(\P):=\{\K(P) \,\vert\, P \text{ is an $r$-polyad of } \P\}$.
\end{definition}
Note that, by Definition~\ref{def:r-partition}, each $r$-partite $r$-graph $F \in \P^{(r)}$ satisfies $F \sub \K(P)$ for a unique $r$-polyad $P$ of $\P$.
In this context, let $\under$ be the function mapping $F$ to $P$.
So for example, if $\P=\P^{(1)}\cup\P^{(2)}\cup\P^{(3)}$ is a $3$-partition, for every bipartite graph $F \in \P^{(2)}$ we have that $\under(F)$ is the pair of vertex classes of $F$;
similarly, for every $3$-partite $3$-graph $F \in \P^{(3)}$ we have that $\under(F)$ is the unique $3$-partite graph of $\P^{(2)}$ whose set of triangles contains the edges of $F$.

We encourage the reader to verify that Definition~\ref{def:r-partition} is indeed compatible with the explicit description of a $1$-, $2$- and $3$-partition given above.

\subsection{$\langle \d \rangle$-regularity of graphs and hypergraphs}

In this subsection we define our new\footnote{For $k=3$, related notions of regularity were studied in~\cite{ReiherRoSc16,Towsner17}.} notion of $\langle\d\rangle$-regularity, first for graphs and then for $k$-graphs for any $k \ge 2$ in Definition~\ref{def:k-reg} below.
\begin{definition}[graph $\langle\d\rangle$-regularity]\label{def:star-regular}
	A bipartite graph $G$ on $(A,B)$ is \emph{$\langle \d \rangle$-regular} if for all subsets $A' \sub A$ and $B' \sub B$ with $|A'| \ge \d|A|$ and $|B'|\ge\d|B|$ we have $d_G(A',B') \ge \frac12 d_G(A,B)$.\\
	A vertex partition $\P$ of a graph $G$ is \emph{$\langle \d \rangle$-regular}
	if one can add/remove at most $\d \cdot e(G)$ edges so that the bipartite graph induced on each $(A,B)$ with $A \neq B \in \P$ is $\langle \d \rangle$-regular.
	%
\end{definition}

For the reader worried that in Definition~\ref{def:star-regular} we merely replaced the $\e$
from the definition of $\e$-regularity with $\d$, we refer to the discussion following Theorem~\ref{theo:main} below.


%
%

The definition of $\langle\d\rangle$-regularity for hypergraphs involves the $\langle\d\rangle$-regularity notion for graphs, applied to certain auxiliary graphs which are defined as follows.
Henceforth, if $P$ is a $(k-1)$-graph and $H$ is $k$-graph then we say that $H$  is \emph{underlain} by $P$ if $H \sub \K(P)$.


\begin{definition}[The auxiliary graph $G_{H}^i$]\label{def:aux}
	For a $k$-partite $k$-graph $H$ on vertex classes $(V_1,\ldots,V_k)$,
	we define a bipartite graph $G_{H}^1$ on the vertex classes $(V_2 \times\cdots\times V_k,\,V_1)$ by
	$$E(G_{H}^1) = \big\{ ((v_2,\ldots,v_k),v_1) \,\big\vert\, (v_1,\ldots,v_k) \in E(H) \big\} \;.$$
	The graphs $G_{H}^i$ for $2 \le i \le k$ are defined in an analogous manner.
%
%
More generally, if $H$ is underlain by the $k$-polyad $P=(F_1,\ldots,F_k)$ then we define $G_{H,P}^i$ as the induced subgraph $G_{H,P}^i=G_{H}^i[F_i,V_i]$.
\end{definition}

As a trivial example, if $H$ is a bipartite graph then $G_H^1$ and $G_H^2$ are both isomorphic to $H$.

Importantly, for a $k$-partition (as defined in Definition~\ref{def:r-partition}) to be $\langle\d\rangle$-regular it must first satisfy a requirement on the regularity of its parts.


\begin{definition}[$\langle\d\rangle$-good partition]\label{def:k-good}
A $k$-partition $\P$ on $V$ is \emph{$\langle\d\rangle$-good} if for every $2 \le r \le k$ and every $F \in \P^{(r)}$ the following holds;
letting $P=\under(F)$ be the $r$-polyad of $\P$ underlying $F$, for every $1 \le i \le r$ the bipartite graph $G_{F,P}^i$ is $\langle \d \rangle$-regular.
\end{definition}
Note that a $1$-partition is trivially $\langle \d \rangle$-good for any $\d$.
Moreover, a $2$-partition $\P$ is $\langle \d \rangle$-good if and only if every bipartite graph in $\P^{(2)}$ (between any two distinct vertex clusters of $\P^{(1)}$) is $\langle \d \rangle$-regular (recall the remark after Definition \ref{def:aux}).




For a $(k-1)$-partition $\P$ with $\P^{(1)} \prec \{V_1,\ldots,V_k\}$ we henceforth denote, for every $1 \le i \le k$,
\begin{equation}\label{eq:partition-notation}
V_i(\P) = \Big\{Z \in \P^{(1)} \,\vert\, Z \sub V_i\Big\} \quad\text{ and }\quad E_i(\P) = \Big\{E \in \P^{(k-1)} \,\vert\, E \sub \prod_{j \neq i} V_j\Big\}.\footnote{$\prod_{j \neq i} V_j = V_1\times\cdots\times V_{i-1}\times V_{i+1}\times\cdots\times V_k$.}
\end{equation}

\begin{definition}[$\langle\d\rangle$-regular partition]\label{def:k-reg}
	Let $H$ be a $k$-partite $k$-graph on vertex classes $(V_1,\ldots,V_k)$
	and $\P$ be a $\langle \d \rangle$-good $(k-1)$-partition with $\P^{(1)} \prec \{V_1,\ldots,V_k\}$.
	We say that $\P$ is a \emph{$\langle \d \rangle$}-regular partition of $H$ if
	for every $1 \le i \le k$,
	$E_i(\P) \cup V_i(\P)$ is a $\langle \d \rangle$-regular partition of $G_H^i$.
	%
%
%
	%
	%
\end{definition}
Note that for $k=2$, Definition~\ref{def:k-reg} reduces to Definition~\ref{def:star-regular}.
For $k=3$, a $\langle \d \rangle$-regular partition of a $3$-partite $3$-graph $H$ on vertex classes $(V_1,V_2,V_3)$ is a $2$-partition $\P=\P^{(1)} \cup \P^{(2)}$ satisfying that: $(i)$ $\P$ is $\langle \d \rangle$-good per Definition~\ref{def:k-good};
$(ii)$ from the auxiliary graph
$G_H^1$, on $(V_2 \times V_3,\,V_1)$,
one can add/remove at most $\d$-fraction of the edges such that for every graph $F \in E_1(\P)$ (so $F \sub V_2 \times V_3$) 
and every vertex cluster $V \in V_1(\P)$ (so $V \sub V_1$), 
the induced bipartite graph $G_H^1[F,V]$ is $\langle \d \rangle$-regular; and~$(iii)$ the analogues of~$(ii)$ in $G_H^2$ and $G_H^3$ hold as well.

\subsection{Formal statement of the main result}\label{subsec:formal}

We are now ready to formally state our Ackermann-type lower bound for $k$-graph $\langle \d \rangle$-regularity (the formal version of
Theorem~\ref{thm:main-informal} above). Recall that we set $\Ack_{1}(x)=2^x$ and then define for every $k \geq 1$ the $(k+1)$-st
Ackermann function $\Ack_{k+1}(n)$ to be the $n$-times composition of $\Ack_{k}(n)$, that is,
$\Ack_{k+1}(n)= \Ack_{k}(\Ack_{k}(\cdots(\Ack_{k}(1))\cdots))$.

\begin{theo}[Main result]\label{theo:main}
	The following holds for every $k \ge 2$ and $s \in \N$.
	There exists a $k$-partite $k$-graph $H$ of density at least $2^{-s-k}$,
	and a partition $\V_0$ of $V(H)$ with $|\V_0| \le 2^{200}k$, such that for every $\langle 2^{-16^k} \rangle$-regular partition $\P$ of $H$, if $\P^{(1)} \prec \V_0$ then $|\P^{(1)}| \ge \Ack_k(s)$.
%
\end{theo}

Let us draw the reader's attention to an important and perhaps surprising aspect of Theorem~\ref{theo:main}.
All the known tower-type lower bounds for graph regularity depend on the error parameter $\epsilon$,
that is, they show the existence of graphs $G$ with the property that every $\epsilon$-regular partition of $G$ is of order at least $\Ack_2(\poly(1/\e))$.
This should be contrasted with the fact that our lower bounds for $\langle \d \rangle$-regularity (for graphs and more generally $k$-graphs) holds for a {\em fixed} error parameter $\delta$.
Indeed, instead of the dependence on the error parameter, our lower bound depends on the {\em density} of the $k$-graph.
This delicate difference makes it possible for us to execute the inductive part of the proof of Theorem~\ref{theo:main}.

\subsection{The core construction and proof overview}\label{subsec:overview}

The graph construction in Lemma~\ref{theo:core} below is the main technical result we will need in order to prove Theorem~\ref{theo:main}.
We stress that the proof of this lemma (which appears in Section~\ref{sec:core}) is completely
independent of the proof of Theorem~\ref{theo:main} (which appears in Section~\ref{sec:LB}).
We will first need to define ``approximate'' refinement (a notion that goes back to Gowers~\cite{Gowers97}).
\begin{definition}[Approximate refinements]
For sets $S$ and $T$ we write $S \sub_\b T$ if $|S \sm T| < \b|S|$.
For a partition $\P$ we write $S \in_\b \P$ if $S \sub_\b P$ for some $P \in \P$.
For partitions $\P$ and $\Q$ of the same set of size $n$ we write $\Q \prec_\b \P$ if
$$\sum_{\substack{Q \in \Q\colon\\Q \notin_\b \P}} |Q| \le \b n \;.$$
\end{definition}
Note that for $\Q$ equitable, $\Q \prec_\b \P$ if and only if
all but at most $\b|\Q|$ parts $Q \in \Q$ satisfy $Q \in_\b \P$.
We note that 
throughout the paper 
we will only use approximate refinements with $\b \le 1/2$, and so if $S \in_\b \P$ then $S \sub_\b P$ for a unique $P \in \P$.

We stress that in Lemma~\ref{theo:core} below we only use notions related to graphs. In particular, $\langle \d \rangle$-regularity refers to Definition~\ref{def:star-regular}.

\begin{lemma}\label{theo:core}
	Let $\Lside$ and $\Rside$ be disjoint sets. Let
	$\L_1 \succ \cdots \succ \L_s$ and $\R_1 \succ \cdots \succ \R_s$ be two sequences of $s$ successively refined equipartitions of $\Lside$ and $\Rside$, respectively,
	that satisfy for  every $i \ge 1$ that:
	\begin{enumerate}
		\item\label{item:core-minR}
		$|\R_i|$ is a power of $2$ and $|\R_1| \ge 2^{200}$,
		\item\label{item:core-expR} $|\R_{i+1}| \ge 4|\R_i|$ if $i < s$,
		\item\label{item:core-expL} $|\L_i| = 2^{|\R_i|/2^{i+10}}$.
	\end{enumerate}
	%
	%
	Then there exists a sequence of $s$ successively refined edge equipartitions $\G_1 \succ \cdots \succ \G_s$ of $\Lside \times \Rside$ such that for every $1 \le j \le s$, $|\G_j|=2^j$, and
	the following holds for every $G \in \G_j$ and $\d \le 2^{-20}$.
	For every $\langle \d \rangle$-regular partition $\P \cup \Q$ of $G$, where $\P$, $\Q$ are partitions of $\Lside$, $\Rside$, respectively, and every $1 \le i \le j$, 	
	if $\Q \prec_{2^{-9}} \R_{i}$ then $\P \prec_{\g} \L_{i}$ with
	$\g = \max\{2^{5}\sqrt{\d},\, 32/\sqrt[6]{|\R_1|} \}$.
\end{lemma}
\begin{remark}
	Every $G \in \G_j$ is a bipartite graph of density $2^{-j}$ since $\G_j$ is equitable.
\end{remark}

An overview of the proof of Lemma~\ref{theo:core} is given in Subsection~\ref{subsec:core-overview}.
Let us end this section by explaining the role Lemma~\ref{theo:core} plays in the proof of Theorem~\ref{theo:main}.

\paragraph{Using graphs to construct $k$-graphs:}
Perhaps the most surprising aspect of the proof of Theorem~\ref{theo:main} is that in order to construct a $k$-graph we use the \emph{graph} construction of Lemma~\ref{theo:core} in a somewhat unexpected way.
In this case, $\Lside$ will be a complete $(k-1)$-partite $(k-1)$-graph and the $\L_i$'s will be partitions of this complete $(k-1)$-graph themselves given by another application of Lemma~\ref{theo:core}.
The size of the partitions will have $\Ack_{k-1}$-type growth, and this application of Lemma~\ref{theo:core} will ``multiply'' the $(k-1)$-graph partitions (given by the $\L_i$'s) to
produce a partition of the complete $k$-partite $k$-graph into $k$-graphs that are hard for $\langle \d \rangle$-regularity.
We will take $H$ in Theorem~\ref{theo:main} to be an arbitrary $k$-graph in this partition.


\paragraph{Why is Lemma~\ref{theo:core} one-sided?}
As is evident from the statement of Lemma~\ref{theo:core}, it is one-sided in nature; that is, under the premise that the partition $\Q$
refines $\R_i$ we may conclude that $\P$ refines $\L_i$.
It is natural to ask if one can do away with this assumption, that is, be able to show that,
under the same assumptions, $\Q$ refines $\R_i$ and $\P$ refines $\L_i$.
As we mentioned in the previous item, 
in order to prove an $\Ack_k$-type lower bound for the $k$-graph regularity lemma we have to apply Lemma~\ref{theo:core} with a sequence of partitions whose size grows as an $\Ack_k$-type function.
%
Now, in this setting, Lemma~\ref{theo:core} does not hold without the one-sided assumption, because if it did, 
then one would have been able to prove (for \emph{every} $k \ge 2$) an $\Ack_k$-type lower bound for
graph $\langle \d \rangle$-regularity, and hence also for Szemer\'edi's regularity lemma.
Put differently, if one wishes to have a construction that holds with arbitrarily fast growing partition sizes, then
one has to introduce the one-sided assumption.

\paragraph{How do we remove the one-sided assumption?}
The proof of Theorem \ref{theo:main} proceeds by first proving a one-sided version of Theorem \ref{theo:main},
stated as Lemma~\ref{lemma:ind-k}. In order to get a construction that does not require such a one-sided assumption,
we will need one final trick; we will take $2k$ clusters of
vertices and arrange $2k$ copies
of this one-sided construction
along the $k$-edges of a cycle. This will give us a ``circle of implications'' that will eliminate the one-sided assumption.
See Subsection~\ref{subsec:pasting}.

\section{Proof of Theorem~\ref{theo:main}}\label{sec:LB}

\renewcommand{\k}{r}

\newcommand{\w}{w}
\renewcommand{\t}{t}

\newcommand{\GG}{\mathbf{G}}
\newcommand{\FF}{\mathbf{F}}
\newcommand{\VV}{\mathbf{V}}

\renewcommand{\Hy}[1]{H_{{#1}}}
\renewcommand{\A}{A}

\newcommand{\subs}{\subset_*}
\newcommand{\pad}{P}
\renewcommand{\K}{\mathcal{K}}
\newcommand{\U}{U}

\renewcommand{\k}{k}

\renewcommand{\K}{\mathcal{K}}
\renewcommand{\r}{k}

The purpose of this section is to prove the main result, Theorem~\ref{theo:main}. 
Its proof crucially relies on a subtle inductive argument (see Lemma~\ref{lemma:ind-k} below).
This section is self-contained save for the application of Lemma~\ref{theo:core}.
The key step of our lower bound proof for $k$-graph regularity, stated as Lemma~\ref{lemma:ind-k} and proved in Subsection~\ref{subsec:main-induction}, relies on a construction that applies Lemma~\ref{theo:core} $k-1$ times. This lemma only gives a ``one-sided'' lower bound, in the spirit of Lemma~\ref{theo:core}.
In Subsection~\ref{subsec:pasting} we show how to use Lemma~\ref{lemma:ind-k} in order to complete the proof of Theorem~\ref{theo:main}.

We first state some properties of $k$-partitions whose proofs are deferred to the end of this section.
The first property relates $\d$-refinements of partitions and $\langle \d \rangle$-regularity of partitions. 
The reader is advised to recall the notation in~(\ref{eq:partition-notation}).


\begin{claim}\label{claim:uniform-refinement}
	Let $\P$ be a $(k-1)$-partition with  $\P^{(1)} \prec \{V_1,\ldots,V_k\}$,
	and let $\F$ be a partition of $V_1\times\cdots\times V_{k-1}$ with $E_k(\P) \prec_\d \F$.
	If $\P$ is $\langle \d \rangle$-good
	then the $(k-2)$-partition $\P'$ obtained by restricting $\P$ to $\bigcup_{i=1}^{k-1} V_i$ is a $\langle 3\d \rangle$-regular partition of some $F \in \F$.
\end{claim}

The second property is given by following easy (but slightly tedious to state) claim.
\begin{claim}\label{claim:restriction}
	Let $H$ be a $k$-partite $k$-graph on vertex classes $(V_1,\ldots,V_k)$, and let $H'$ be the induced $k$-partite $k$-graph on vertex classes $(V_1',\ldots,V_k')$ with $V_i' \sub V_i$ and $\b \cdot e(H)$ edges.
	If $\P$ is a $\langle \d \rangle$-regular partition of $H$ with
	$\P^{(1)} \prec \bigcup_{i=1}^k \{V_i,\,V_i \sm V_i'\}$
	then its restriction $\P'$ to $V(H')$ is a $\langle \d/\b \rangle$-regular partition of $H'$.
\end{claim}

\renewcommand{\K}{k}

\subsection{Key inductive proof}\label{subsec:main-induction}

\paragraph*{Set-up.}

We next introduce a few more definitions that are needed for the statement of Lemma~\ref{lemma:ind-k}.
Let $e(i) = 2^{i+10}$. We define the following tower-type function $\t\colon\N\to\N$;
\begin{equation}\label{eq:t}
\t(i+1) = \begin{cases}
2^{\t(i)/e(i)}	&\text{if } i \ge 1\\
2^{200}	&\text{if } i = 0 \;.
\end{cases}
\end{equation}
It is easy to prove, by induction on $i$, that $\t(i) \ge e(i)\t(i-1)$ for $i \ge 2$ (for the induction step,
$t(i+1) \ge 2^{\t(i-1)} = t(i)^{e(i-1)}$, so $t(i+1)/e(i+1) \ge \t(i)^{e(i-1)-i-11} \ge \t(i)$).
This means that $t$ is a monotone increasing function, and that it is always an integer power of $2$ (follows by induction as $t(i)/e(i) \ge 1$ is a positive power of $2$ and in particular an integer).
We record the following facts regarding $\t$ for later use:
\begin{equation}\label{eq:monotone}
\t(i) \ge 4\t(i-1)  \quad\text{ and }\quad
\text{ $\t(i)$ is a power of $2$} \;.
\end{equation}
For a function $f:\N\to\N$ with $f(i) \ge i$ we denote
\begin{equation}\label{eq:f*}
f^*(i) = \t\big(f(i)\big)/e(i) \;.
\end{equation}
Note that $f^*(i)$ is indeed a positive integer (by the monotonicity of $\t$, $f^*(i) \ge \t(i)/e(i)$ is a positive power of $2$).
In fact, $f^*(i) \ge f(i)$ (as $f^*(i) \ge 4^{f(i)}/e(i)$ using~(\ref{eq:monotone})).

We recursively define the function $\A_\k\colon\N\to\N$ for any integer $\k \ge 2$ as follows: $\A_2(i)=i$, whereas for $\k \ge 3$, 
\begin{equation}\label{eq:Ak}
\A_\k(i+1) = \begin{cases}
\A_{\k-1}(\A_\k^*(i))		&\text{if } i \ge 1\\
2^{2^{3\K+2}}	&\text{if } i = 0 \;.
\end{cases}
\end{equation}
Note that $\A_\k$ is well defined since $\A_\k^*(i) \in \N$ for $i \ge 1$ and $k \ge 2$, and that $\A_\k^{},\A_\k^*$ are both monotone increasing.
It is evident that $\A_\k$ grows like the $k$-th level Ackermann function; in fact, one can check that for every $\k \ge 2$ we have
\begin{equation}\label{eq:A_k}
\A_\k(i) \ge \Ack_\k(i) \;.
\end{equation}
Furthermore, we denote, for $\k \ge 1$,
\begin{equation}\label{eq:delta_k}
\d_\k = 2^{-8^\k} \;.
\end{equation}
We moreover denote, for $\k \ge 2$,
\begin{equation}\label{eq:m_k}
m_\k(i):=\A_2^*(\cdots(\A^*_{\k}(i))\cdots) \;.
\end{equation}
%
We next record a few easy bounds for later use.
Recall that
\begin{equation}\label{eq:delta_12}
\d_1 = 2^{-8} \quad\text{ and }\quad \d_2 = 2^{-64} \;.
\end{equation}
Noting the relation $\d_{\k} = \d_{\k-1}^8$, we have for $\k \ge 3$ that
\begin{equation}\label{eq:delta_k-bound}
\d_\k^{1/4} = \d_{\k-1}^2 \le \d_2 \d_{\k-1} = 2^{-64}\d_{\k-1} \;.
\end{equation}
Noting the relation $\A_\K(1) = \d_\K^{-4}$ for $\K \ge 3$, we have for $\K \ge 3$ that
\begin{equation}\label{eq:t1-bound}
1/\sqrt[6]{\A_\K(1)} \le \d_{\K}^{1/2} = \d_{\K-1}^{4} \le \d_1^3\d_{\K-1} = 2^{-24}\d_{\K-1} \;.
\end{equation}


\paragraph*{Key inductive proof.}

The key argument in our lower bound proof for $\K$-graph regularity is the following result, which is proved by induction on the hypergraph's uniformity.





\begin{lemma}[$\k$-graph induction]\label{lemma:ind-k}
	Let $s \in \N$, let $\Vside^1,\ldots,\Vside^\k$ be $\k \ge 2$ mutually disjoint sets of equal size and let $\V_1 \succ\cdots\succ \V_m$ be a sequence of
	$m=m_\k(s)$ 
	successive equitable refinements of $\{\Vside^1,\ldots,\Vside^\k\}$
	with $|V_h(\V_i)|=\t(i)$ for every $i,h$.\footnote{Since we assume that each $\V_i$ refines $\{\Vside^1,\ldots,\Vside^k\}$ then $V_h(\V_i)$ is 
		(recall~(\ref{eq:partition-notation}))
		the restriction of $\V_i$ to $\Vside^h$.}
	Then there exists a sequence of $s$ successively refined equipartitions $\HH_1 \succ \cdots \succ \HH_s$ of $\Vside^1 \times \cdots \times \Vside^\k$ such that for every $1 \le j \le s$, $|\HH_j|=2^j$
	and every $H \in \HH_j$ satisfies the following property:\\
		If $\P$ is a $\langle \d_\k \rangle$-regular
		partition of $H$, and for some $1 \le i \le \A_\k(j)$ $(<m)$ we have $V_h(\P) \prec_{2^{-9}} V_h(\V_i)$ for every $2 \le h \le \k$, then $V_1(\P) \prec_{2^{-9}} V_1(\V_{i+1})$.
\end{lemma}
Note that $\V_i$ is well defined in the property described in Lemma~\ref{lemma:ind-k} since $i \le \A_\k(j) \le m$.
\begin{proof}	
	We proceed by induction on $\k \ge 2$.
	For the induction basis $\k=2$ we are given $s \in \N$, two disjoint sets $\Vside^1,\Vside^2$ 
	as well as $m=\A_2^*(s)$ ($\ge s+1$) successive equitable refinements $\V_1 \succ\cdots\succ \V_{m}$ of $\{\Vside^1,\Vside^2\}$.
	Our goal is to find a sequence of $s$ successively refined equipartitions $\HH_1 \succ \cdots \succ \HH_s$ of $\Vside^1 \times \Vside^2$ as in the statement.
	To prove the induction basis, apply Lemma~\ref{theo:core} with
	$$\Lside=\Vside^1,\quad \Rside=\Vside^2 \quad\text{ and }\quad V_1(\V_2) \succ \cdots \succ V_1(\V_{s+1}) ,\quad V_2(\V_1) \succ \cdots \succ V_2(\V_{s}) \;,$$
	and let
	$$\G^1 \succ \cdots \succ \G^{s} \quad\text{ with }\quad |G^\ell|=2^\ell \text{ for every } 1 \le \ell \le s $$
	be the resulting sequence of $s$ successively refined equipartitions of $\Vside^1 \times \Vside^2$.
	These two sequences indeed satisfy assumptions~\ref{item:core-minR},~\ref{item:core-expR} in Lemma~\ref{theo:core} since $|V_2(\V_j)|=\t(j)$ and by~(\ref{eq:monotone});
	moreover, they satisfy assumption~\ref{item:core-expL} since for every $1 \le j \le s$ we have
	$$|V_1(\V_{j+1})| = \t(j+1) = 2^{\t(j)/e(j)} = 2^{|V_2(\V_{j})|/e(j)} \;,$$
	where the second equality uses the definition of the function $\t$ in~(\ref{eq:t}).
	We will show that taking $\HH_j=\G_j$ for every $1 \le j \le s$ yields a sequence as required by the statement.
	Fix $1 \le j \le s$ and $G \in \G_j$; note that $G$ is a bipartite graph on the vertex classes $(\Vside^1,\Vside^2)$.
	Moreover, let $1 \le i \le j$ (recall $\A_2(j)=j$) and let $\P$ be a $\langle \d_2 \rangle$-regular 
	partition of $G$
	with $V_2(\P) \prec_{2^{-9}} V_2(\V_{i})$.
	Since $\d_2 \le 2^{-20}$ by~(\ref{eq:delta_12}),
	Lemma~\ref{theo:core} implies that $V_1(\P) \prec_{x} V_1(\V_{i+1})$ with $x=\max\{2^{5}\sqrt{\d_2},\, 32/\sqrt[6]{\t(1)} \}$.
	Using~(\ref{eq:delta_12}) and~(\ref{eq:t})
	we have $x \le 2^{-9}$, completing the proof of the induction basis.
	
	To prove the induction step, recall that we are given $s \in \N$, disjoint sets $\Vside^1,\ldots,\Vside^\k$ 
	and a sequence of $m=m_\k(s)$ successive equitable refinements $\V_1 \succ\cdots\succ \V_m$ of $\{\Vside^1,\ldots,\Vside^\k\}$,
	and our goal is to construct a sequence of $s$ successively refined equipartitions $\HH_1 \succ \cdots \succ \HH_s$ of $\Vside^1 \times \cdots \times \Vside^\k$ as in the statement.
	We begin by applying the induction hypothesis with $\k-1$ (which would imply Proposition~\ref{prop:main-k-hypo} below).
	We henceforth put $c=2^{-9}$.
	Now, apply the induction hypothesis with $\k-1$ on
	\begin{equation}\label{eq:ind-hypo}
	\text{
		$s':=\A_\k^*(s)$ (in place of $s$),\, $\Vside^1,\ldots,\Vside^{\k-1}$ and $\bigcup_{h=1}^{\k-1}V_h(\V_1) \succ\cdots\succ \bigcup_{h=1}^{\k-1}V_h(\V_m)$} \;,
	\end{equation}
	and let
	\begin{equation}\label{eq:main-k-colors}
	\F_1 \succ \cdots \succ \F_{s'} \quad\text{ with }\quad |\F_ \ell|=2^\ell \text{ for every } 1 \le \ell \le s'
	\end{equation}
	be the resulting sequence of $s'$ successively refined equipartitions of $\Vside^1 \times \cdots \times \Vside^{\k-1}$.
	\begin{prop}[Induction hypothesis]\label{prop:main-k-hypo}
		Let $1 \le \ell \le s'$ and $F \in \F_\ell$.
		If $\P'$ is a $\langle \d_{\k-1} \rangle$-regular partition of $F$ with ${\P'}^{(1)} \prec \{\Vside^1,\ldots,\Vside^{\k-1}\}$, and for some $1 \le i \le \A_{\k-1}(\ell)$ $(<m_{k-1}(s))$ we have
		$V_h(\P') \prec_c V_h(\V_i)$ for every $2 \le h \le \k-1$,
		then $V_1(\P) \prec_c V_1(\V_{i+1})$.
	\end{prop}
	\begin{proof}
		It suffices to verify that the number $m$ of partitions in~(\ref{eq:ind-hypo}) is as required by the induction hypothesis. Indeed, by~(\ref{eq:m_k}),
		$$m_{\k-1}(s') = \A_2^*(\cdots(\A_{\k-1}^*(s'))\cdots) = \A_2^*(\cdots(\A^*_{\k}(s))\cdots) = m_\k(s) = m \;.$$	
	\end{proof}
For each $1 \le j \le s$ let
\begin{equation}\label{eq:main-k-dfns}
\F_{(j)} = \F_{\A_k^*(j)}
\quad\text{ and }\quad
\V_{(j)} = V_k(\V_{\A_k(j)}) \;.
\end{equation}	
All these choices are well defined since $\A_\k^*(j)$ satisfies $1 \le \A_\k^*(1) \le \A_\k^*(j) \le \A_\k^*(s) = s'$ by our choice of $s'$ in~(\ref{eq:ind-hypo}), and since $\A_\k(j)$ satisfies $1 \le \A_\k(1) \le \A_\k(j) \le \A_\k(s) \le m$.
Observe that we have thus chosen two subsequences of $\F_1,\cdots,\F_{s'}$ and $V_k(\V_1),\ldots,V_k(\V_m)$, each of length $s$.
Recalling that each $\F_{(j)}$ is a partition of $\Vside^1 \times\cdots\times \Vside^{k-1}$, apply Lemma~\ref{theo:core} with
$$
\Lside=\Vside^1 \times\cdots\times \Vside^{k-1},\quad \Rside=\Vside^k \quad\text{ and }\quad \F_{(1)} \succ \cdots \succ \F_{(s)}, \quad \V_{(1)} \succ \cdots \succ \V_{(s)} \;,
$$
and let
\begin{equation}\label{eq:ind-colors2}
\G_1 \succ \cdots \succ \G_{s} \quad\text{ with }\quad |\G_\ell|=2^\ell \text{ for every } 1 \le \ell \le s
\end{equation}
be the resulting sequence of $s$ successively refined graph equipartitions of $(\Vside^1\times\cdots\times\Vside^{\k-1})\times\Vside^\k$.

\begin{prop}[Core proposition]\label{prop:ind-prop2}
	Let $1 \le j \le s$ and $G \in \G_j$.
	If $\E \cup \V$ is a $\langle \d_\k \rangle$-regular partition of $G$ (where $\E$ and $\V$ are partitions of $\Vside^1\times\cdots\times\Vside^{\k-1}$ and $\Vside^\k$ respectively),
	and for some $1 \le j' \le j$ we have $\V \prec_{c} \V_{(j')}$,
	then $\E \prec_{\frac14\d_{\k-1}} \F_{(j')}$.
\end{prop}
\begin{proof}
	First we need to verify that we may apply Lemma~\ref{theo:core} as above.
	Assumptions~\ref{item:core-minR} and~\ref{item:core-expR} follow from the fact that $|\V_{(j)}|=\t(\A_\k(j))$, from~(\ref{eq:monotone})
	and the fact that $\A_k(1) \ge 2^{2^{11}} \ge 2^{200}$ for $\k \ge 3$ by~(\ref{eq:Ak}).	
	To verify that assumption~\ref{item:core-expL} holds, note that $|\F_{(j)}|=2^{\A_\k^*(j)}$
	by~(\ref{eq:main-k-colors}),
	and that
	$|\V_{(j)}|=\t(\A_\k(j))$
	by the statement's assumption that 
	$|V_k(\V_i)|=\t(i)$. 
	Thus, indeed,
	$$
	|\F_{(j)}| = 2^{\A_\k^*(j)} = 2^{\t(\A_\k(j))/e(j)} = 2^{|\V_{(j)}|/e(j)} \;,
	$$
	where the second equality uses the definition in~(\ref{eq:f*}).
	%
	Moreover, note that $\d_\k \le \d_2 \le 2^{-20}$ by~(\ref{eq:delta_12}).
	We can thus use Lemma~\ref{theo:core} to infer that the fact that $\V \prec_{c} \V_{(j')}$ implies that $\E \prec_x \F_{(j')}$ with $x=\max\{2^{5}\sqrt{\d_\k},\, 32/\sqrt[6]{\t(\A_\K(1))} \}$.
	To see that indeed $x \le \frac14\d_{\k-1}$,
	apply~(\ref{eq:delta_k-bound}) as well as the fact that $\t(\A_\K(1)) \ge \A_\K(1)$ and~(\ref{eq:t1-bound}).
\end{proof}

	For each $G \in \G_j$ let $\Hy{G}$ be the $k$-partite $k$-graph on vertex classes $(\Vside^1,\ldots,\Vside^k)$ with edge set
	$$E(\Hy{G}) = \big\{ (v_1,\ldots,v_k) \,:\, ((v_1,\ldots,v_{k-1}),v_k) \in E(G) \big\} \;,$$
	and note that we have (recall Definition \ref{def:aux})
	\begin{equation}\label{eqH}
	G=G_{\Hy{G}}^k\;.
	\end{equation}
	For every $1 \le j \le s$ let $\HH_j=\{\Hy{G} \,\vert\, G \in \G_j \}$, and note that $|\HH_j|=|\G_j|=2^j$ by~(\ref{eq:ind-colors2}),
	that $H_j$ is an equipartition of $\Vside^1\times\cdots\times\Vside^k$,
	and that $\HH_1 \succ\cdots\succ \HH_s$.
	Our goal is to show that these partitions satisfy the property guaranteed by the statement.
	
	Henceforth fix $1 \le j \le s$ and $H \in \HH_j$, and write $H=\Hy{G}$ with $G \in \G_j$.
	To complete the proof is suffices to show that $H$ satisfies the property in the statement.
	Assume now that $i$ is such that
	\begin{equation}\label{eq:ind-i-assumption}
	1 \le i \le \A_\k(j)
	\end{equation}
	and
	\begin{enumerate}
		\item\label{item:ind-reg}
		$\P$ is a $\langle \d_\k \rangle$-regular partition of $H$,
		\item\label{item:ind-refine} $V_h(\P) \prec_{c} V_h(\V_i)$ for every $2 \le h \le \k$.
	\end{enumerate}	
	In the remainder of the proof we will complete the induction step by showing that
	\begin{equation}\label{eq:ind-goal}
	V_1(\P) \prec_{c} V_1(\V_{i+1}) \;.
	\end{equation}
	It follows from Item~\ref{item:ind-reg},
	by Definition~\ref{def:k-reg} and~(\ref{eqH}), that in particular
	\begin{equation}\label{eq:ind-reg}
	E_k(\P) \cup V_k(\P) \text{ is a $\langle \d_\k \rangle$-regular partition of } G.
	\end{equation}
		Let
	\begin{equation}\label{eq:ind-j'}
	1 \le j' \le s
	\end{equation}
	be the unique integer satisfying
	\begin{equation}\label{eq:ind-sandwich}
	\A_k(j') \le i < \A_k(j'+1) \;.
	\end{equation}
	Note that (\ref{eq:ind-j'}) holds due to~(\ref{eq:ind-i-assumption}).
	Recalling~(\ref{eq:main-k-dfns}),
	the lower bound in~(\ref{eq:ind-sandwich}) implies that $V_k(\V^i) \prec V_k(\V_{\A_k(j')}) = \V_{(j')}$.
	Therefore, the assumption $V_k(\P) \prec_{c} V_k(\V^i)$ in Item~\ref{item:ind-refine} implies that
	\begin{equation}\label{eq:ind-Zk}
	V_k(\P) \prec_{c} \V_{(j')} \;.
	\end{equation}
	Apply Proposition~\ref{prop:ind-prop2} on $G$, using~(\ref{eq:ind-reg}),~(\ref{eq:ind-j'}) and~(\ref{eq:ind-Zk}), to deduce that
	\begin{equation}\label{eq:ind-E}
	E_k(\P) \prec_{\frac14\d_{k-1}} \F_{(j')} = \F_{\A_k^*(j')} \;,
	\end{equation}
	where for the equality again recall~(\ref{eq:main-k-dfns}).
	Let $\P^*$ be the restriction of $\P$ to $\Vside^1\cup\cdots\cup\Vside^{\k-1}$,
	and let $\P' = \P^* \sm \P^*[\Vside^1\times\cdots\times\Vside^{\k-1}]$.
	Note that $\P^*$ is a $(\k-1)$-partition on $(\Vside^1,\ldots,\Vside^{\k-1})$ and that $\P'$ is a $(\k-2)$-partition on $(\Vside^1,\ldots,\Vside^{\k-1})$.
	Since $\P$ is a $\langle \d_k \rangle$-regular partition of $H$ (by Item~\ref{item:ind-reg} above), $\P^*$
	is in particular $\langle \d_k \rangle$-good. By~(\ref{eq:ind-E}) we may thus apply Claim~\ref{claim:uniform-refinement} on $\P^*$
	to conclude that
	\begin{equation}\label{eq:ind-reg2}
	\P' \text{ is a }
	\langle \d_{k-1} \rangle
	 \text{-regular partition of some $F\in\F_{\A_k^*(j')}$.}
	\end{equation}
	By~(\ref{eq:ind-reg2}) we may apply Proposition~\ref{prop:main-k-hypo} with $F$, $\P'$, $\ell=\A_k^*(j')$ and $i$, observing (crucially)
	that $i \leq \ell$ by (\ref{eq:ind-sandwich}).
	Note that Item~\ref{item:ind-refine} in particular implies that $V_h(\P') \prec_{c} V_h(\V_i)$ for every $2 \le h \le k-1$.
	We thus deduce that $V_1(\P') \prec_{c} V_1(\V_{i+1})$.
	Since $V_1(\P') = V_1(\P)$, this proves~(\ref{eq:ind-goal}) and thus completes the induction step and the proof of Lemma~\ref{lemma:ind-k}.
\end{proof}

\subsection{Putting everything together}\label{subsec:pasting}

We can now prove our main theorem, Theorem~\ref{theo:main}, which we repeat here for convenience.
\addtocounter{theo}{-1}
\begin{theo}[Main theorem]\label{theo:main}
	The following holds for every $k \ge 2$ and $s \in \N$.
	There exists a $k$-partite $k$-graph $H$ of density at least $2^{-s-k}$,
	and a partition $\V_0$ of $V(H)$ with $|\V_0| \le 2^{200}k$, such that if $\P$ is a $\langle 2^{-16^k} \rangle$-regular partition of $H$ with $\P^{(1)} \prec \V_0$ then $|\P^{(1)}| \ge \Ack_k(s)$.
%
%
	
%
%
\end{theo}

\begin{remark}\label{remark:main}
	As can be easily checked, the proof of Theorem~\ref{theo:main} also gives that $H$ has the same number of vertices in all vertex classes.
\end{remark}


\begin{proof}
	Let the $k$-graph $B$ be the tight $2k$-cycle; that is, $B$ is the $k$-graph on vertex classes $\{0,1,\ldots,2k-1\}$ with edge set $E(B)=\{\{0,1,\ldots,k-1\},\{1,2,\ldots,k\},\ldots,\{2k-1,0,\ldots,k-2\}\}$.
	Note that $B$ is $k$-partite with vertex classes $(\{0,k\},\{1,k+1\},\ldots,\{k-1,2k-1\}\}$.
	Put $m=m_k(s-k)$ and let $n \ge \t(m)$.
	Let $\Vside^0,\ldots,\Vside^{2k-1}$ be $2k$ mutually disjoint sets of size $n$ each.
	Let $\V^1 \succ\cdots\succ \V^m$ be an arbitrary sequence of $m$ successive equitable refinements of $\{\Vside^0,\ldots,\Vside^{2k-1}\}$ with $|\V^i_h|=\t(i)$ for every $1 \le i \le m$ and $0 \le h \le 2k-1$, which exists as $n$ is large enough.
	Extending the notation $\V_x$ (above Definition~\ref{def:k-reg}), for every $0 \le x \le 2k-1$ we henceforth denote the restriction of the vertex partition $\V \prec \{\Vside^0,\ldots,\Vside^{2k-1}\}$ to $\Vside^x$ by $\V_x = \{V \in \V \,\vert\, V \sub \Vside^x\}$.
	For each edge $e=\{x,x+1,\ldots,x+k-1\} \in E(B)$ (here and henceforth when specifying an edge, the integers are implicitly taken modulo $2k$)
	apply Lemma~\ref{lemma:ind-k} with
	$$s,\,
	\Vside^{x},\Vside^{x+1},\ldots,\Vside^{x+k-1}
	\text{ and }
	\bigcup_{j=0}^{k-1}\V^{1}_{x+j}
	\succ\cdots\succ \bigcup_{j=0}^{k-1}\V^{m}_{x+j} \;.$$
	Let $H_e$ denote the resulting $k$-partite $k$-graph on $(\Vside^{x},\Vside^{x+1},\ldots,\Vside^{x+k-1})$.
	Note that $d(H_e) = 2^{-s}$.
	Let
	$$c = 2^{-9} \quad\text{ and }\quad K=\A_k(s)+1 \;.$$
	Then $H_e$ has the property that for every $\langle \d_k \rangle$-regular partition $\P'$ of $H_e$ and every $1 \le i < K$,
	\begin{equation}\label{eq:paste-property}
	\text{If $V_{x+h}(\P') \prec_{c} V_{x+h}(\V_i)$ for every $1 \le h \le k-1$, then $V_x(\P) \prec_{c} V_x(\V_{i+1})$.}
	\end{equation}	
	We construct our $k$-graph on the vertex set $\Vside:=\Vside^0 \cup\cdots\cup \Vside^{2k-1}$ as
	$E(H) = \bigcup_{e} E(H_e)$; that is, $H$ is the edge-disjoint union of all $2k$ $k$-partite $k$-graphs $H_e$ constructed above.
	Note that $H$ is a $k$-partite $k$-graph (on vertex classes $(\Vside^0 \cup \Vside^k,\, \Vside^1 \cup \Vside^{k+1},\ldots, \Vside^{k-1} \cup \Vside^{2k-1}))$ of density $\frac{2k}{2^k} 2^{-s} \ge 2^{-s-k}$, as needed.
	We will later use the following fact.
	\begin{prop}\label{prop:restriction}
		Let $\P$ be a $\langle 2^{-16^k} \rangle$-regular partition of $H$ with $\P^{(1)} \prec \{\Vside^0,\ldots,\Vside^{2k-1}\}$, and let
		$e \in E(B)$.
		Then the restriction $\P'$ of $\P$ to $V(H_e)$ is a $\langle \d_k \rangle$-regular partition of $H_e$.
	\end{prop}
	\begin{proof}
		Immediate from Claim~\ref{claim:restriction} using the fact that $e(H_e) = \frac{1}{2k}e(H)$.
		%
	\end{proof}	
	
	Now, let $\P$ be an $\langle 2^{-16^k} \rangle$-regular partition of $H$
	with $\P^{(1)} \prec \V^1$.
	Our goal will be to show that
	\begin{equation}\label{eq:paste-goal}
	\P^{(1)} \prec_{c} \V^{K} \;.
	\end{equation}
	Proving~(\ref{eq:paste-goal}) would complete the proof, by setting $\V_0$ in the statement to be $\V^1$ here (notice $|\V^1|=k\t(1) = k2^{200}$ by~(\ref{eq:t})); 
	indeed, Claim~\ref{claim:refinement-size}, given in Section~\ref{subsec:pre} below, would imply that
	$$|\P^{(1)}| \ge \frac14|\V^{K}| = \frac14 \cdot 2k \cdot \t(K)
	\ge \t(K)
	\ge \t(\A_k(s))
	\ge \A_k(s)
	\ge \Ack_k(s) \;,$$
	where the last inequality uses~$(\ref{eq:A_k})$.
	Assume towards contradiction that $\P^{(1)} \nprec_{c} \V^{K}$. By averaging,
	\begin{equation}\label{eq:assumption}
	\P^{(1)}_h \nprec_c \V^{K}_h \text{ for some } 0 \le h \le 2k-1.
	\end{equation}
	For each $0 \le h \le 2k-1$ let $1 \le \b(h) \le K$ be the largest integer satisfying $\P^{(1)}_h \prec_c \V^{\b(h)}_h$,
	which is well defined since $\P^{(1)}_h \prec_c \V^1_h$ (in fact $\P^{(1)} \prec \V^1$).
	Put $\b^* = \min_{0 \le h \le 2k-1} \b(h)$, and note that by~(\ref{eq:assumption}),
	\begin{equation}\label{eq:paste-star}
	\b^* < K \;.
	\end{equation}
	Let $0 \le x \le 2k-1$ minimize $\b$, that is, $\b(x)=\b^*$.
	Therefore:
	\begin{equation}\label{eqcontra}
	\P^{(1)}_{x+k-1} \prec_c \V^{\b^*}_{x+k-1},\,\ldots,\,\P^{(1)}_{x+1} \prec_c \V^{\b^*}_{x+1} \text{ and }
	\P^{(1)}_{x} \nprec_c \V^{\b^*+1}_{x}.	
	\end{equation}
	Let $e=\{x,x+1,\ldots,x+k-1\} \in E(B)$.
	Let $\P'$ be the restriction of $\P$ to $V(H_e)=\Vside^{x} \cup \Vside^{x+1} \cup\cdots\cup \Vside^{x+k-1}$, which is a $\langle \d_k \rangle$-regular partition of $H_e$ by Proposition~\ref{prop:restriction}. Since ${\P'}^{(x+h)}_{h}=\P^{(x+h)}_h$ for every $0 \le h \le k-1$ we get
	from~(\ref{eqcontra}) a contradiction to~(\ref{eq:paste-property}) with $i=\beta^*$.
	We have thus proved~(\ref{eq:paste-goal}) and so the proof is complete.
\end{proof}

\subsection{Deferred proofs: properties of $k$-partitions}\label{subsec:k-partitions-proofs}

\renewcommand{\K}{\mathcal{K}}

Henceforth, for a $(k-1)$-partite $(k-1)$-graph $F$ on $(V_1,\ldots,V_{k-1})$ and a disjoint vertex set $V$ we denote by $F \circ V$ the $k$-partite $k$-graph on $(V_1,\ldots,V_{k-1},V)$ given by
$$F \circ V := \{ (v_1,\ldots,v_{k}) \,\vert\, (v_1,\ldots,v_{k-1}) \in F \text{ and } v_{k} \in V \} \;.$$
%
We will use the following additional property of $k$-partitions.

\begin{claim}\label{claim:decomposition}
	Let $\P$ be a $(k-1)$-partition 
	with $\P^{(1)} \prec (\Vside^1,\ldots,\Vside^k)$,
	$F \in E_k(\P)$ and $V \in V_k(\P)$.
	Then there is a set of $k$-polyads $\{P_i\}_i$ of $\P$ such that
	\begin{equation}\label{eq:red-k-partitionP-gen}
	F \circ V = \bigcup_i \K(P_i) \,\text{ is a partition of $F \circ V$,
		with } P_i = (P_{i,1},\ldots,P_{i,k-1},F) \;.
	\end{equation}
\end{claim}

\begin{proof}
	We proceed by induction on $\k \ge 2$, noting that the induction basis $\k=2$ is trivial since in this case $E_2(\P) = V_1(\P)$ so
	$F=V' \in V_1(\P)$,
	hence $F \circ V = V' \times V$ is simply $\K(P)$ where $P$ is the $2$-polyad of $\P$ corresponding to the pair $(V',V)$.
	For the induction step assume the statement holds for $k \ge 2$ and let us prove it for $k+1$. Let $\P$ be a $k$-partition on $(\Vside^1,\ldots,\Vside^{k+1})$, let $F \in E_{k+1}(\P)$ and let $V \in V_{k+1}(\P)$,
	and denote the vertex classes of $F$ by $(V_1,\ldots,V_k)$ with $V_j \sub \Vside^j$ for every $1 \le j \le k$.
	Recall that, by Definition~\ref{def:r-partition},
	$\P^{(\k)} \prec \K_{\k}(\P)$.
	Thus, $F \sub \K(G_1,\ldots,G_\k)$ with $G_j \in \P^{(\k-1)}$ for every $1 \le j \le \k$, where $G_j$ is a $(\k-1)$-partite $(\k-1)$-graph on $(V_1,\ldots,V_{j-1},V_{j+1},\ldots,V_\k)$.
	We have
	\begin{equation}\label{eq:decompose}
	F \circ V = \K(G_1 \circ V, \ldots,\, G_{\k} \circ V,\, F) \;;
	\end{equation}
	indeed, the  inclusion~$\sub$ follows from the fact that $F \sub \K(G_1,\ldots,G_\k)$, and the reverse inclusion~$\supseteq$ is immediate.
	Now, for every $1 \le j \le \k$, let $\P_j$ denote the restriction of $\P$ to the vertex classes $(\Vside^1,\ldots,\Vside^{j-1},\Vside^{j+1},\ldots,\Vside^\k,\Vside^{k+1})$
	and apply the induction hypothesis with the $(\k-1)$-partition $\P_j$, the $(k-1)$-graph $G_j$ and $V$. 
	It follows that there is a partition $G_j \circ V = \bigcup_i \K(P_{j,i})$ where each $P_{j,i}$ is a $\k$-polyad of $\P_j$ (and thus of $\P$) on $(V_1,\ldots,V_{j-1},V_{j+1},\ldots,V_\k,V)$.
	Since $\K_{\k}(\P_j) \succ \P_j^{(\k)}$, again by Definition~\ref{def:r-partition},
	for each $i$ and $j$ we have a partition $\K(P_{j,i}) = \bigcup_\ell F_{j,i,\ell}$ with $F_{j,i,\ell} \in \P_j^{(\k)}$, where each $F_{j,i,\ell}$ is a $\k$-partite $\k$-graph on $(V_1,\ldots,V_{j-1},V_{j+1},\ldots,V_\k,V)$.
	Summarizing, for every $1 \le j \le k$ we have the partition $G_j \circ V = \bigcup_{i,\ell} F_{j,i,\ell}$,
	and so it follows using~(\ref{eq:decompose}) that we have the partition
	$$
	F \circ V = \K\Big(\bigcup_{i,\ell} F_{1,i,\ell}, \ldots,\, \bigcup_{i,\ell} F_{\k,i,\ell},\, F \Big)
	= \bigcup_{\substack{i_1,\ldots,i_\k\\\ell_1,\ldots,\ell_\k}} \K(F_{1,i,\ell},\ldots,\,F_{\k,i,\ell},\, F) \;.
	$$
	As each $(\k+1)$-tuple $(F_{1,i,\ell},\ldots,\,F_{\k,i,\ell}, F)$ corresponds to a $(\k+1)$-polyad of $\P$, this completes the inductive step.
\end{proof}

Before proving Claim~\ref{claim:uniform-refinement} we will also need the following two easy claims.
\begin{claim}\label{claim:star-union}
	Let $G_1,\ldots,G_\ell$ be mutually edge-disjoint bipartite graphs on the same vertex classes $(Z,Z')$.
	If every $G_i$ is $\langle \d \rangle$-regular then $G=\bigcup_{i=1}^\ell G_i$ is also $\langle \d \rangle$-regular.
\end{claim}
\begin{proof}
	Let $S \sub Z$, $S' \sub Z'$ with $|S| \ge \d|Z|$, $|S'| \ge \d|Z'|$.
	Then
	$$d_G(S,S') = \frac{e_G(S,S')}{|S||S'|}
	= \sum_{i=1}^\ell \frac{e_{G_i}(S,S')}{|S||S'|}
	= \sum_{i=1}^\ell d_{G_i}(S,S')
	\ge \sum_{i=1}^\ell \frac12 d_{G_i}(Z,Z')
	= \frac12 d_{G}(Z,Z') \;,$$
	where the second and last equalities follow from the mutual disjointness of the $G_i$, and the inequality follows from the $\langle \d \rangle$-regularity of each $G_i$.
	Thus, $G$ is $\langle \d \rangle$-regular, as claimed.
\end{proof}

\begin{claim}\label{claim:refinement-union}
	If $\Q \prec_\d \P$ then there exist $P \in \P$ and $Q$ that is a union of members of $\Q$ such that $|P \triangle Q| \le 3\d|P|$.
\end{claim}
\begin{proof}
	For each $P\in \P$ let $\Q(P) = \{Q \in \Q \,\vert\, Q \sub_\d P\}$,
	and denote $P_\Q = \bigcup_{Q \in \Q(P)} Q$.
	We have
	\begin{align*}
	\sum_{P \in \P} |P \triangle P_\Q|
	&= \sum_{P \in \P} |P_\Q \sm P| + \sum_{P \in \P} |P \sm P_\Q|
	= \sum_{P \in \P} \sum_{\substack{Q \in \Q \colon\\Q \sub_\d P}} |Q \sm P|
	+ \sum_{P \in \P} \sum_{\substack{Q \in \Q \colon\\Q \nsubseteq_\d P}} |Q \cap P| \\
	&\le \sum_{P \in \P} \sum_{\substack{Q \in \Q \colon\\Q \sub_\d P}} \d|Q|
	+ \Big( \sum_{\substack{Q \in \Q\colon\\Q \notin_\d \P}} |Q|
	+ \sum_{\substack{Q \in \Q \colon\\Q \in_\d \P}} \d|Q| \Big)
	\le 3\d\sum_{Q \in \Q} |Q|
	= 3\d\sum_{P \in \P} |P| \;,
	\end{align*}
	where the last inequality uses the statement's assumption $\Q \prec_\d \P$ to bound the middle summand.
	By averaging, there exists $P \in \P$ such that $|P \triangle P_\Q| \le 3\d|P|$, thus completing the proof.
\end{proof}

\paragraph*{Proofs of properties.}
We are now ready to prove the properties of $k$-partitions stated at the beginning of Section~\ref{sec:LB}.

\renewcommand{\k}{r}

\begin{proof}[Proof of Claim~\ref{claim:uniform-refinement}]
	Put $\E = E_k(\P)$, and let us henceforth use $\k=k-1$. Since $\E \prec_\d \F$, Claim~\ref{claim:refinement-union} implies that there exist $F \in \F$ (an $\k$-partite $\k$-graph on $(\Vside^1,\ldots,\Vside^{\k})$), as well as an $\k$-partite $\k$-graph $F_\E$ that is a union of members of $\E$, such that $|F \triangle F_\E| \le 3\d|F|$.
	Denote by $\Q$ the $(\k-1)$-partition $\P'$ obtained by restricting $\P$ to $\bigcup_{i=1}^{\k} V_i$, and note that $\Q$ is $\langle \d \rangle$-good since $\Q \sub \P$ and, by assumption, $\P$ is $\langle \d \rangle$-good.
	Our goal is to prove that $\Q$ is a $\langle 3\d \rangle$-regular partition of $F$.
	Recalling Definition~\ref{def:k-reg}, note that it suffices to show, without loss of generality, that $E_{\k}(\Q) \cup V_{\k}(\Q)$ is a $\langle \d \rangle$-regular partition of the bipartite graph $G_{F}^{\k}$. 
	We have $|G_{F}^{\k} \triangle G_{F_\E}^\k|=|F \triangle F_\E| \le 3\d|F|=3\d|G_{F}^{\k}|$, that is, $G_{F_\E}^\k$ is obtained from $G_{F}^{\k}$ by adding/removing at most $3\d |G_{F}^{\k}|$ edges.
	Therefore, to complete the proof it suffices to show that for every $Z \in E_{\k}(\Q)$ and $Z' \in V_{\k}(\Q)$, the induced bipartite graph $G_{F_\E}^\k[Z,Z']$ is $\langle \d \rangle$-regular (recall Definition~\ref{def:star-regular}).
	
	Apply Claim~\ref{claim:decomposition} on the $(\k-1)$-partition $\Q$ with $Z$ and $Z'$. Since $\E \prec \K_{\r-1}(\Q)$ (recall Definition~\ref{def:r-partition}),
	this means that $Z \circ Z'$ is a (disjoint) union of members $E$ of $\E$ all underlain by $\k$-polyads of the form  $(P_1,\ldots,P_{\k-1},Z)$.
	Since $\Q$ is $\langle \d \rangle$-good (recall Definition~\ref{def:k-good}), for each such $E$ we in particular have that $G^\k_E[Z,Z']$ ($=G^\k_{E,\,\under(E)}$) is $\langle \d \rangle$-regular.
	It follows that $G_{F_\E}^\k[Z,Z']$ is a disjoint union of $\langle \d \rangle$-regular bipartite graphs on $(Z,Z')$.
	Claim~\ref{claim:star-union} thus implies that $G_{F_\E}^\k[Z,Z']$ is a $\langle \d \rangle$-regular bipartite graph. As explained above, this completes the proof.
\end{proof}

We end this subsection with the easy proof of Claim~\ref{claim:restriction}.
\begin{proof}[Proof of Claim~\ref{claim:restriction}]
	Recall Definition~\ref{def:k-reg}.
	Clearly, $\P'$ is $\langle \d \rangle$-good. We will show that $E_1(\P') \cup V_1(\P')$ is a $\langle \d/\b \rangle$-regular partition of $G^1_{H'}$.
	The argument for $G^i_{H'}$ for every $2 \le i \le k$ will be analogous, hence the proof would follow.
	Observe that $G^1_{H'}$ is an induced subgraph of $G^1_{H}$, namely, $G^1_{H'} = G^i_{H}[V_2' \times\cdots\times V_k',\, V_1']$.
	By assumption, $e(H') = \b e(H)$, and thus $e(G^1_{H'}) = \b e(G^1_{H})$.
	By the statement's assumption on $\P^{(1)}$ and since $E_1(\P) \cup V_1(\P)$ is a $\langle \d \rangle$-regular partition of $G^1_{H}$, we deduce---by adding/removing at most $\d e(G^1_{H}) =  (\d/\b)e(G^1_{H'})$ edges of $G^1_{H'}$---that $E_1(\P') \cup V_1(\P')$ is a $\langle \d/\b \rangle$-regular partition of $G_{H'}^1$.
	As explained above, this completes the proof.
\end{proof}




\section{Proof of Lemma~\ref{theo:core}}\label{sec:core}


%
The purpose of this section is to prove Lemma~\ref{theo:core}.
This section, which deals exclusively with graphs (rather than hypergraphs), is organized as follows.
In Subsection~\ref{subsec:core-overview} we give a short overview of our construction and analysis with pointers to the main claims in this section, in Subsection~\ref{subsec:pre} we introduce a few auxiliary notions and prove a few simple lemmas, in Subsection~\ref{subsec:core-construction} we specify our construction, in Subsection~\ref{subsec:key-properties} we prove two key properties of our construction and in Subsection~\ref{subsec:LB-core-main} we prove Lemma~\ref{theo:core}.

\subsection{Proof overview}\label{subsec:core-overview}

\paragraph*{Construction.}
In our construction of the sequence of edge equitpartition $\G_1 \succ \cdots \succ \G_s$ of $\Lside \times \Rside$ for Lemma~\ref{theo:core}, each graph in each $\G_j$ is a blowup of a bipartite graph whose vertex classes are the partitions $(\L_j,\R_j)$. 
To obtain $\G_j$, we subdivide each graph in $\G_{j-1}$ (where $\G_0$ is the trivial partition $\{\Lside \times \Rside\}$ with one part) into two subgraphs with the same number of edges. The way we do this is, roughly speaking,
blowup the graph from $\L_{j-1} \cup \R_{j-1}$ into $\L_j \cup \R_j$ (i.e., replace every ``vertex'' in $\L_{j-1} \cup \R_{j-1}$ by a cluster in $\L_j \cup \R_j$ and replace every ``edge'' by a complete bipartite graph),
and subdivide each of the complete bipartite graphs into a random biregular bipartite subgraph of density $\frac12$ and its complement (see Claim~\ref{claim:core-properties} below).

\paragraph*{Analysis.}
In order to prove that every graph $G \in \G_j$ satisfies the guarantee of Lemma~\ref{theo:core} we use two key properties of our construction.
The first key property (Claim~\ref{claim:property_qr} below) is that $G$ is \emph{somewhat} quasirandom, where
the measure of quasirandomness depends inversely on $|\R_1|$.
The second key property (Claim~\ref{claim:prop-main} below) is essentially that 
for every subset $P \sub \Lside$ satisfying $\P \sub_\d L \in \L_{k-1}$ yet $P \notin_{\d} \L_k$ for some $k \le j$, a constant fraction of the clusters $R \in \R_k$ ``neighboring'' $L$ in $G$ are such that in order to make the induced bipartite graph $G[P,R]$ $\langle \d \rangle$-regular one has to add or remove a constant fraction of its edges.
Recall that what we need to prove in Lemma~\ref{theo:core} is, in the contrapositive, 
that a vertex partition $\P \cup \Q$ of $G$ ($\P$ partitions $\Lside$, $\Q$ partitions $\Rside$) satisfying $\Q \prec_{2^{-9}} \Q_i$ yet $\P \nprec_{\g} \L_i$---with $\g=\poly(\d,|\R_1|^{-1})$ and some $i \le j$---cannot be $\langle \d \rangle$-regular.
Applying the second key property on all clusters $P \in \P$, \emph{each with its appropriate $k$},  
allows us to deduce, using the first key property, that in order for all induced bipartite graphs $G[P,Q]$, with $P \in \P$ and $Q \in \Q$, to be $\langle \d \rangle$-regular, one has to add or remove more than a $\d$-fraction of the edges of $G$. This proves that $\P \cup \Q$ is not a $\langle \d \rangle$-regular partition of $G$, as needed.

\subsection{Construction preliminaries}\label{subsec:pre}

In this subsection we prove a sequence of simple lemmas that will be used later in this section.



\paragraph*{Approximate refinements.}

We will need the following claim. 
\begin{claim}\label{claim:refinement-size}
	If $\Q \prec_{1/2} \P$ and $\P$ is equitable then $|\Q| \ge \frac14|\P|$.
\end{claim}
\begin{proof}
	We claim that the underlying set $U$ has a subset $U^*$ of size $|U^*|\ge \frac14|U|$ such that the partitions $\Q^*=\{Q \cap U^* \,\vert\, Q \in \Q \} \setminus \{\emptyset\}$ and $\P^*=\{P \cap U^* \,\vert\, P \in \P \} \setminus \{\emptyset\}$
	of $U^*$ satisfy $\Q^* \prec \P^*$.
	Indeed, let $U^* = \bigcup_{Q} Q \cap P_Q$ where the union is over all $Q \in \Q$ satisfying $Q \sub_{1/2} P_Q$ for a (unique) $P_Q \in \P$.
	As claimed, $|U^*| = \sum_{Q \in_{1/2} \P} |Q \cap P_Q| \ge \sum_{Q \in_{1/2} \P} \frac12|Q| \ge \frac14|U|$, using $\Q \prec_{1/2} \P$ for the last inequality.
	Now, since $\P$ is equitable, $|\P^*| \ge \frac14|\P|$.
	Thus, $|\Q| \ge |\Q^*| \ge |\P^*| \ge \frac14|\P|$, as desired.
\end{proof}

\paragraph*{Quasirandom graphs.}

A bipartite graph $G=(V,U,E)$ of density $p$ is  \emph{$(\e)$-regular} if for all sets $A \sub U$, $B \sub V$ with $|S| \ge \e|V|$, $|T| \ge \e|U|$ we have
$$|d(S,T)-p| \le \e p \;.$$

The next lemma---which, importantly, applies to arbitrarily sparse graphs---will be later used to prove that the graph we construct is quasirandom by bounding from above the average codegree.
Formally, $\codeg(v,v')$ is the number of vertices that neighbor both $v$ and $v'$.

\begin{lemma}\label{lemma:codegree}
	Let $G=(V,U,E)$ be a biregular bipartite graph of density $p$
	where for every $v \in V$,
	$$\sum_{v' \in V} \max\{\codeg(v,v') - p^2|U|,\,0\} \le \a \cdot p^2 |U||V| \;.$$
	Then $G$ is $(2\a^{1/6})$-regular.
\end{lemma}
\begin{proof}
	%
	Put $\e=2\a^{1/6}$.
	Let $S \sub V$, $T \sub U$ with $|S| \ge \e|V|$, $|T| \ge \e|U|$. Our goal is to show that $d(S,T) = (1\pm \e)p$.
	Let $u$ be a uniformly random vertex from $U$, and let the random variable $D$ be the degree of $u$ into $S$, that is, $D=e(S,u)$.
	Then
	$$\Ex[D] = \frac{1}{|U|}\sum_{u \in U} e(S,u) = \frac{1}{|U|}\sum_{v \in S} \deg_G(v) = p|S| \;,$$
	where the last equality uses the biregularity of $G$.
	Furthermore,
	$$\Ex[D^2] = \frac{1}{|U|}\sum_{u \in U} e(S,u)^2
	= \frac{1}{|U|} \sum_{v,v' \in S} \codeg(v,v') \;.$$
	Therefore,	
	\begin{align*}
	\Var[D] &= \frac{1}{|U|}\sum_{v,v' \in S} \codeg(v,v') - (p|S|)^2
	= \frac{1}{|U|}\sum_{v,v' \in S} \big(\codeg(v,v') - p^2|U|\big)\\
	&\le \frac{1}{|U|}\sum_{v \in S} \sum_{v' \in V} \max\{\codeg(v,v') - p^2|U|,\,0\}
	\le \a p^2 |S||V| \;,
	\end{align*}
	where the last inequality follows from the statement's assumption.
	Set, with foresight, $x=\a^{1/3}/\e$.
	By Chebyshev's inequality,
	$$\Pr\Big(\big|D-p|S|\big| \ge xp|S|\Big) \le \frac{\Var[D]}{(xp|S|)^2}
	\le \frac{\a}{x^2\e} = \e^2 x \;.$$
	Since $e(S,T) = \sum_{u \in T} e(S,u)$, we have
	$$(|T|-\e^2x|U|) \cdot p|S|(1-x)
	\le e(S,T)
	\le |T|\cdot p|S|(1+x) + \e^2x|U| \cdot p|V|\;,$$
	where the right inequality again uses the biregularity of $G$.
	Therefore,
	$$
	1-2x \le (1-\e x)(1-x)
	\le \frac{d(S,T)}{p}
	\le (1+x) + x = 1+2x \;.$$
	Since $2x \le \e$ by our choice of $\e=2\a^{1/6}$, the proof follows.
\end{proof}

\newcommand{\Xside}{\mathbf{X}}
\newcommand{\Yside}{\mathbf{Y}}
\newcommand{\Ups}{\mathbf{\Upsilon}}

\newcommand{\Ubold}{\mathbf{U}}
\renewcommand{\U}{\mathcal{U}}

%

%


\paragraph*{Balanced graphs.}
Here we define another notion of quasirandom bipartite graphs.
We say that a bipartite graph $G$ on (an ordered pair) $(\Xside,\Yside)$ is\footnote{That is to say, this definition is not symmetric with respect $\Xside$,$\Yside$.} \emph{$\b$-balanced} if for every $x \neq x' \in \Xside$, the number of vertices in $\Yside$ that are either neighbors of $x,x'$ or non-neighbors of $x,x'$ is at most $(\frac12 + \b)|\Yside|$.
We say that $G$ is \emph{complement-closed} if there is an involution $\phi:\Yside\to\Yside$ such that $N(\phi(y))=\Xside \sm N(y)$ for every $y \in \Yside$.

\begin{definition}\label{def:balanced-graph}
	Let $\Gamma$ be a bipartite graph on $(\Xside,\Yside)$, let $\X$ be a partition of $\Xside$, let $\Y$ be a partition of $\Yside$, and let $\F$ be a family of subsets of $\Yside$ where each member of $\F$ is a union of parts of $\Y$.
	We say that $\Gamma$ is $(\X,\Y,\F,\a,\b)$-balanced if the following four conditions hold:
	\begin{enumerate}	
		\item(Equitable) for every $X \in \X$ and $y \in \Yside$ we have $|N(y) \cap X| = \frac12|X|$,
		\item($\b$-balanced) for every $F \in \F$ the graph $\Gamma[\Xside,F]$ is $\b$-balanced,
		\item(Pseudorandom) for every $X \in \X$, $F \in \F$ and $y \neq y' \in F$,
		$|N(y) \cap N(y') \cap X| \le (1+\a)\frac14|X|$,
		\item(Complement-closed) for every
		$Y \in \Y$ the graph $\Gamma[\Xside,Y]$ is complement-closed.
	\end{enumerate}
\end{definition}

\begin{remark}\label{remark:complement}
	Condition~$(iv)$ in Definition~\ref{def:balanced-graph} in particular implies that for every
	$F \in \F$ the induced bipartite graph $\Gamma[\Xside,F]$ is complement-closed, as $F$ is assumed to be a union of parts of $\Y$.
\end{remark}
\begin{remark}
	Any balanced graph $\Gamma$ is biregular, since the vertices in $\Yside$ all have degree $\frac12|\Xside|$ by condition~$(i)$, and the vertices in $\Xside$ all have degree $\frac12|\Yside|$ by condition~$(iv)$.
\end{remark}

We note that condition~$(ii)$ of Definition~\ref{def:balanced-graph} can be shown to follow from condition~$(iii)$ for a suitable $\a=\a(\b)$ by using Claim~\ref{lemma:codegree}, but we opted for including the two conditions for clarity.

\begin{lemma}\label{lemma:seq-exists}
	Let $\Xside,\Yside$ be disjoint sets, let $\X$ be partition of $\Xside$ into parts each of even size $m$, let $\Y$ be a partition of $\Yside$ into parts each of even size,
	and let $\F$ be a family of subsets of $\Yside$ where each is of size $k$ and is a union of parts of $\Y$.
	Let $\a \ge m^{-1/3}$ and suppose $2^{21} \le k \le m$ and $\max\{|\X|,|\F|\} \le m \le 2^{k/600}$.
	Then there exists an $(\X,\Y,\F,\a,\frac{1}{16})$-balanced bipartite graph $\Gamma$ on $(\Xside,\Yside)$.
\end{lemma}

\newcommand{\YYside}{\Yside_1}
\renewcommand{\FF}{\mathcal{F}_1}
\newcommand{\YY}{\mathcal{Y}_1}
\newcommand{\XX}{\mathcal{X}}

\begin{proof}
	The proof idea is quite simple---an appropriate  random bipartite graph satisfies the first three conditions in Definition~\ref{def:balanced-graph}, and ``closing'' it under complementation additionally guarantees the last condition.	
	The precise details follow.
	
	For each part $Y \in \Y$ let $Y=Y_1 \cup Y_2$ be an arbitrary bipartition (recall that $|Y|$ is even).
	For each $i \in \{1,2\}$ put $\Y_i = \{Y_i \colon Y \in \Y\}$ and $\Yside_i = \bigcup_{Y \in \Y} Y_i$,
	and let $\phi:\Yside_2\to\Yside_1$ be a bijection mapping  $Y_2$ to $Y_1$ for each $Y \in \Y$.
	Furthermore, for each $F \in \F$ let $F_i = F \cap \Yside_i$, and put $\F_i = \{ F_i \,\vert\, F \in \F\}$.
	Since every member of $\F$ is a union of parts of $\Y$ and of size $k$,
	we have that each member of $\F_1$ and of $\F_2$ is of size $k/2$, with $\phi$ mapping $F_2$ to $F_1$ for each $F \in \F$.
	We first prove that there is a bipartite graph $\Gamma_1$ on $(\Xside,\Yside_1)$ satisfying the following three conditions:
	\begin{enumerate}	
		\item[$(A)$] for every $X \in \XX$ and $y \in \YYside$ we have $|N(y) \cap X| = \frac12|X|$,
		\item[$(B)$] for every $F \in \F_1$ the graph $\Gamma[\Xside,F]$ is $\b$-balanced,
		\item[$(C)$] for every $X \in \XX$, $F \in \FF$ and $y \neq y' \in F$ we have
		$|N(y) \cap N(y') \cap X| = (1\pm\a)\frac14|X|$.
	\end{enumerate}
	We construct $\Gamma_1$ in a random fashion by choosing, independently for each $X \in \XX$ and $y \in \YYside$, precisely $|X|/2$ uniformly random neighbors of $y$ in $X$ (recall $|X|=m$ is even).
	It suffices to show that $\Gamma_1$ satisfies conditions~$(B)$ and $(C)$ with positive probability.
	To prove~$(B)$, fix 
	$F \in \FF$ and $x \neq x' \in \Xside$.	
	For every $y \in F$, the probability that either $x,x' \in N(y)$ or $x,x' \notin N(y)$ is easily seen to be at most $1/2$.
	By Chernoff's inequality,\footnote{We use the basic version stating that $\Pr[\text{Bin}(n,p) \ge pn+t] \le \exp(-2t^2/n)$; see, e.g.,~\cite{Heoffding63}.} the number of $y \in F$ satisfying the above is at most $(\frac12 + \frac{1}{16})|F|$ except with probability at most $\exp(-2|F|/(16)^2) = \exp(-|F|/2^7)$.
	Therefore, the probability that condition~$(B)$ does not hold is, by the union bound, at most
	$$|\FF|\binom{|\Xside|}{2} \cdot \exp\Big(-\frac{(k/2)}{2^7}\Big)
	\le \big(2^{k/600}\big)^3 \cdot \exp\Big(-\frac{k}{2^8}\Big)
	< \frac12 \;,$$
	where the last inequality uses the assumption $k \ge 2^{21}$.
	To prove~$(C)$, fix $X \in \XX$, $F \in \FF$ and $y \neq y' \in F$.
	By construction, the number of neighbors of $y'$ that lie in $N(y) \cap X$ follows a hypergeometric distribution.
	Thus, we may apply the Chernoff bound (see, e.g., Section~6 in~\cite{Heoffding63} or Theorem~2.10 in~\cite{JansonLuRu11}) and deduce that
	$|N(y) \cap N(y') \cap X| = (\frac12 \pm \frac12\a)|N(y) \cap X|$ $(= \frac14(1 \pm \a)|X|)$ 	
	except with probability at most $2\exp(-2(\frac14\a)^2|X|) = 2\exp(-\a^2|X|/8)$.
	Therefore, the probability that condition~$(C)$ does not hold is, by the union bound, at most
	$$|\XX||\FF|\binom{k/2}{2} \cdot 2\exp\Big(-\frac18\a^2m\Big)
	\le m^4 \cdot \exp\Big(-\frac18 m^{1/3}\Big)
	< \frac12 \;,$$
	where the last inequality uses the assumption $m \ge 2^{21}$.
	We deduce from the above, again by the union bound, that $\Gamma$ satisfies conditions~$(B)$ and~$(C)$ with positive probability, proving the existence of $\Gamma_1$ as desired.
	
	Consider now the isomorphic copy of $\Gamma_1$ on the vertex classes $(\Xside,\Yside_2)$ that is determined by $\phi$,
	and let $\Gamma_2$ be its complement.
	We claim that $\Gamma_2$ also satisfies conditions~$(A)$, $(B)$ and $(C)$, with $\Yside_2$ and $\F_2$ replacing $\Yside_1$ and $\F_1$, respectively.
	Indeed, condition~$(A)$ clearly holds, condition~$(B)$ follows from the symmetry in the definition of a $\b$-balanced graph, and one can easily verify that condition~$(C)$ follows from conditions~$(A)$ and $(C)$ satisfied by $\Gamma_1$ 
	(using $N_{\Gamma_2}(y) \cap N_{\Gamma_2}(y') \cap X = X \sm (N_{\Gamma_1}(\phi^{-1}(y)) \cup N_{\Gamma_1}(\phi^{-1}(y')))$ and the inclusion-exclusion formula).
	
	To complete the proof we claim that the (edge-disjoint) union $\Gamma = \Gamma_1 \cup \Gamma_2$ is an $(\X,\Y,\F,\a,\frac{1}{16})$-balanced graph on $(\Xside,\Yside)$, as needed.
	Note that condition~$(iv)$ is immediate from the definition of $\Gamma_2$, condition~$(i)$ follows from~$(A)$,
	and condition~$(ii)$ follows from $(B)$ since $\Gamma$ is a union of $\b$-balanced graphs sharing the vertex class $\Xside$.
	As for condition~$(iii)$, let $X \in \X$, $F \in \F$ and $y \neq y' \in F$.
	If $y,y' \in \Yside_1$ or $y,y' \in \Yside_2$ then we are done by condition~$(C)$ which is satisfied by $\Gamma_1$ and $\Gamma_2$.
	Hence, suppose $y \in \Yside_1$ while $y' \in \Yside_2$.
	Observe that $N_\Gamma(y') = \Xside \sm N_{\Gamma_1}(\phi(y'))$. Thus, if $\phi(y')=y$ then $|N_\Gamma(y) \cap N_\Gamma(y')|=0$ which completes the proof, and
	otherwise, since $\Gamma_1$ satisfies conditions~$(A)$ and $(C)$, we complete the proof as in the end of the previous paragraph.
\end{proof}

\subsection{Core construction}\label{subsec:core-construction}

\newcommand{\Rf}{U}
\newcommand{\Nf}{\mathcal{N}_i}
\newcommand{\Df}{D}

We now turn to define our construction for Lemma~\ref{theo:core}.
Let $\Lside$ and $\Rside$ be disjoint sets, not necessarily of the same size.
Let $\L_1 \succ \cdots \succ \L_s$ and $\R_1 \succ \cdots \succ \R_s$ each be a sequence of $s$ successively
refined equipartitions of $\Lside$ and $\Rside$, respectively, satisfying the three conditions in Lemma~\ref{theo:core}.
We construct a sequence of $s$ successively refined edge equipartitions $\G_1 \succ \cdots \succ \G_s$ of $\Lside \times \Rside$ with $|\G_i|=2^i$.
For convenience, set $\G_0=\{\Lside\times\Rside\}$ to be the trivial partition, which consists only of the complete bipartite graph on $\Lside,\Rside$.
The construction is iterative, where for every $i \ge 1$ the partition $\G_i$ will be obtained from $\G_{i-1}$ by subdividing the edge set of each graph $G_{i-1} \in \G_{i-1}$ into two biregular
bipartite graphs of density $2^{-i}$ (note that this property is satisfied by $\G_0$).\footnote{To be clear, all our graphs are bipartite on the vertex classes $(\Lside,\Rside)$.}
Furthermore, $G_i$ will have the property that $d_{G_i}(L,R) \in \{0,1\}$ for every $L \in \L_i$ and $R \in \R_i$. We can thus define for every $G_i \in \G_i$ a bipartite graph $\widetilde{G}_i$ on the vertex classes $(\L_i,\R_i)$,
where $(L,R) \in E(\widetilde{G}_i)$ if and only if $d_{G_i}(L,R)=1$.
In other words, the graph $G_i$ is a blowup\footnote{This in particular means that for (the single graph) $G_0 \in \G_0$ we have that $\widetilde{G}_0$ is $K_{1,1}$, i.e., the single-edge graph.}
of $\widetilde{G}_i$. In particular, we will frequently refer to a part $L \in \L_i$ or $R \in \R_i$ both as
a cluster of vertices of $G_i$ or as a vertex of $\widetilde{G}_i$.
%

Fix an $i \geq 1$ and one of the graphs $G_{i-1} \in \G_{i-1}$.
We will now show how to partition $G_{i-1}$ (more precisely $E(G_{i-1})$) into
two biregular bipartite graphs $G_i$ and $G_{i-1} \setminus G_i$, both of density $2^{-i}$.
It will be more convenient to (indirectly) define $G_i$ by defining $\widetilde{G}_i$.
We henceforth abbreviate $N_{\widetilde{G}_{i-1}}(\cdot)$ by $N_{i-1}(\cdot)$.
This means, for example, that if $R \in \R_{i-1}$ and $L \in \L_{i-1}$ and $R \in N_{i-1}(L)$ then $G_{i-1}$ contains the complete bipartite graph on $(L,R)$.
We henceforth denote $\P[X] = \{ P \in \P \,\vert\, P \sub X\}$ the restriction of a partition $\P$ to a set $X$.

For $L \in \L_{i-1}$ let $\Nf(L)$ be the family of clusters of $\R_i$ contained in a cluster of $\R_{i-1}$ adjacent to $L$ in $\widetilde{G}_{i-1}$; more formally,
\begin{equation}\label{eq:noation-frak}
\Nf(L):=\bigcup_{R \in N_{i-1}(L)} \R_i[R] \;.
\end{equation}
Note that, since $\widetilde{G}_{i-1}$ is biregular of density $2^{i-1}$,
we have
\begin{equation}\label{eq:N(L)}
|\Nf(L)| = d(\widetilde{G}_{i-1})|\R_{i-1}| \cdot \frac{|\R_i|}{|\R_{i-1}|} = d(G_{i-1})|\R_i| = \frac{|\R_i|}{2^{i-1}} \;.
\end{equation}

We will next use Lemma~\ref{lemma:seq-exists} in order to find a bipartite graph $\Gamma_i$ on the vertex classes $(\L_i,\R_i)$ that is $(\X_i,\Y_i,\F_i,\a_i,\frac{1}{16})$-balanced, where
$\X_i = \{ \L_i[L] \,\vert\, L \in \L_{i-1} \}$,
$\Y_i = \{ \R_i[R] \,\vert\, R \in \R_{i-1} \}$,
$\F_i = \{ \Nf(L) \,\vert\, L \in \L_{i-1}\}$ and
\begin{equation}\label{eq:alpha}
\a_i = |\L_i|^{-1/6} \;.
\end{equation}
In the rest of this paragraph we explain why we can indeed apply Lemma~\ref{lemma:seq-exists}.
We will use the assumptions~\ref{item:core-minR},~\ref{item:core-expR} and~\ref{item:core-expL} of Lemma~\ref{theo:core}, and in particular the following inequality which they imply:
\begin{equation}\label{eq:n_i-bound}
|\L_i| = 2^{|\R_{i}|/2^{i+10}} \ge 2^{2|\R_{i-1}|/2^{i+9}} = |\L_{i-1}|^2 \;.
\end{equation}
To prove our claim we observe the following; $\F_i$ is $k_i$-uniform with $k_i=|\R_{i}|/2^{i-1}$ by~(\ref{eq:N(L)});
the members of $\X_i$ and of $\Y_i$ are of size $m_i=|\L_i|/|\L_{i-1}|$ and $|\R_i|/|\R_{i-1}|$, respectively, with both numbers being powers of $2$ and thus even; since $m_i \ge \sqrt{|\L_i|}$ by~(\ref{eq:n_i-bound}), we have $2^{21} \le k_i \le m_i$ (by item~\ref{item:core-expL} of Lemma~\ref{theo:core}) and $\a_i = |\L_i|^{-1/6} \ge m_i^{-1/3}$;
$|\X_i|=|\F_i|=|\L_{i-1}| \le \sqrt{|\L_i|} \le m_i$;
and, finally, $m_i \le |\L_i| \le 2^{k_i/600}$.
We thus conclude that the above graph $\Gamma_i$ indeed exists.

%

Recall that our intention is to construct a graph $\widetilde{G}_i$ on vertex classes $(\L_i,\R_i)$ and that the above-constructed $\Gamma_i$ is also a graph on $(\L_i,\R_i)$.
We define $\widetilde{G}_i$ as the intersection of $\widetilde{G}_{i-1}$ and $\Gamma_i$. More precisely,
\begin{equation}\label{eq:core-construction}
E(\widetilde{G}_i) = \{\, (L',R') \in E(\Gamma_i) \,\vert\, L'\sub L, R' \sub R \text{ with } (L,R) \in E(\widetilde{G}_{i-1}) \} \;.
\end{equation}
Having defined $\widetilde{G}_i$ we have thus defined a graph $G_i \subseteq G_{i-1}$. 
In order to obtain $\G_i$ from $\G_{i-1}$ we simply repeat the above for every $G_{i-1} \in \G_{i-1}$, that is, we set
$$\G_i=\bigcup_{G_{i-1} \in \G_{i-1}} \{G_i,\, G_{i-1} \setminus G_i\} \;.$$

Fix $G_i \in \G_i$.
The following claim summarizes two important properties of
$G_i$.
\begin{claim}\label{claim:core-properties}
	Let $L \in \L_{i-1}$.
	Then
	\begin{enumerate}
		\item\label{item:equitable} $G_i[L,R]$ is biregular of density $\frac12$ for every $R \in N_{i-1}(L)$ (so $G_i$ is biregular of density $2^{-i}$),
		\item\label{item:codeg} $\big|N_{i}(R_1') \cap N_{i}(R_2') \cap \L_i[L]\big| \le
		(1+\a_i)\frac14\big|\L_i[L]\big|$  
		for every $R_1' \neq R_2' \in \Nf(L)$.
	\end{enumerate}
\end{claim}
\begin{proof}
	Since the induced bipartite graph $G_{i-1}[L,R]$ is complete, by~(\ref{eq:core-construction}) the induced bipartite graph $\widetilde{G}_i[L,R]$ is precisely $\Gamma_i[L,R]$.
	Conditions~$(i)$ and~$(iv)$ in Definition~\ref{def:balanced-graph} imply the biregularity of $\widetilde{G}_i[L,R]$. Recalling that $G_i$ is a blow-up of $\widetilde{G}_i$, and since $G_{i-1}$ is biregular of density $2^{-(i-1)}$,  property~$(i)$ follows.
	Property~$(ii)$ immediately follows from condition~$(iii)$ in Definition~\ref{def:balanced-graph} and our choice of $\Gamma_i$.
\end{proof}

We furthermore have the following property coming from the balancedness of $\Gamma_i$.

\begin{claim}\label{claim:1-12}	
	Let $L \in \L_{i-1}$.
	For every $\lambda=(\lambda_{L'})_{L' \in \L_i}$ with $\lambda_{L'} \ge 0$ and $\norm{\lambda}_{1}=1$, at least $\frac{1}{6}2^{-i}|\R_i|$ of the clusters $R' \in \Nf(L)$ satisfy that
	\begin{equation}\label{eq:1-12-statement}
	\sum_{L' \notin N_i(R')} \lambda_{L'} \ge \frac18(1-\norm{\lambda}_{\infty}) \quad\text{ and }\quad
	\sum_{\substack{L' \in N_i(R')\colon\\L' \sub L}} \lambda_{L'} \ge \frac12 -
	\sum_{\substack{L' \in \L_i \colon\\L' \nsubseteq L}} \lambda_{L'} \;.
	\end{equation}
\end{claim}


For the proof of Claim~\ref{claim:1-12}	we will need the following lemma from~\cite{MoshkovitzSh16} (which improved upon a similar lemma from~\cite{Gowers97}).
\begin{lemma}[restatement of Lemma~2.3 in~\cite{MoshkovitzSh16}]\label{lemma:1-6}
	Let $G$ be a bipartite graph on $(\Xside,\Yside)$ that is $\frac{1}{16}$-balanced. For every $\lambda=(\lambda_1,\ldots,\lambda_{|\Xside|})$ with $\lambda_x \ge 0$ and $\norm{\lambda}_{1}=1$, at least $\frac16|\Yside|$ of the vertices $y \in \Yside$ satisfy
	$$\min\Big\{\sum_{x\in N_G(y)} \lambda_x,\,\sum_{x \notin N_G(y)} \lambda_x\Big\} \ge \frac18(1-\norm{\lambda}_{\infty}) \;.$$
	%
	%
\end{lemma}

\begin{proof}[Proof of Claim~\ref{claim:1-12}]
	We first observe that $\Gamma_i[\L_i,\,\Nf(L)]$ is $\frac{1}{16}$-balanced and complement-closed;
	this
	follows from conditions $(ii)$, $(iv)$ in Definition~\ref{def:balanced-graph}
	and our choice of $\Gamma_i$ (for~$(iv)$ recall Remark~\ref{remark:complement}).
	By Lemma~\ref{lemma:1-6} (without needing the complement-closed requirement) there are at least $\frac16|\Nf(L)|$ clusters $R' \in \Nf(L)$ satisfying that
	\begin{equation}\label{eq:1-12-min}
	\min\Big\{\sum_{L'\in N_{\Gamma_i}(R')} \lambda_{L'}, \sum_{L' \notin N_{\Gamma_i}(R')} \lambda_{L'}\Big\}
	\ge \frac18(1-\norm{\lambda}_{\infty}) \;.
	\end{equation}
	Now, observe that if $R'$ satisfies~(\ref{eq:1-12-min}) but does not satisfy
	\begin{equation}\label{eq:1-12-half}
	\sum_{L' \in N_{\Gamma_i}(R')} \lambda_{L'} \ge \frac12
	\end{equation}
	then its ``complementary counterpart'' $\phi(R')$ does satisfy both~(\ref{eq:1-12-min}) and~(\ref{eq:1-12-half}).
	This means that at least $\frac{1}{12}|\Nf(L)|$ clusters $R'\in\Nf(L)$ satisfy both~(\ref{eq:1-12-min}) and~(\ref{eq:1-12-half}).
	We claim that each of these $R'$ satisfies~(\ref{eq:1-12-statement}).
	Indeed, by construction~(\ref{eq:core-construction}), if $L' \notin N_{\Gamma_i}(R')$ then $L' \notin N_{i}(R')$, hence
	$$\sum_{L' \notin N_{i}(R')} \lambda_{L'} \ge \sum_{L' \notin N_{\Gamma_i}(R')} \lambda_{L'} \ge \frac18(1-\norm{\lambda}_{\infty}) \;.$$
	Moreover, by~(\ref{eq:core-construction}), since $R' \in \Nf(L)$, we have for every $L' \sub L$ that $L' \in N_{\Gamma_i}(R')$ if and only if $L' \in N_{i}(R')$. Hence, by~(\ref{eq:1-12-half}),
	$$\sum_{\substack{L' \in N_i(R')\colon\\L' \sub L}} \lambda_{L'}
	= \sum_{\substack{L' \in N_{\Gamma_i}(R')\colon\\L' \sub L}} \lambda_{L'}
	= \sum_{L' \in N_{\Gamma_i}(R')} \lambda_{L'} - \sum_{\substack{L' \in N_{\Gamma_i}(R') \colon\\L' \nsubseteq L}} \lambda_{L'}
	\ge \frac12 - \sum_{\substack{L' \in \L_i \colon\\L' \nsubseteq L}} \lambda_{L'} \;.$$
	This proves our claim.
	Since $\frac{1}{12}|\Nf(L)| = \frac{1}{6}2^{-i}|\R_i|$ by~(\ref{eq:N(L)}), the proof follows.
\end{proof}

Having collected some simple properties, in the following subsection we prove two key properties of our construction which will be used to prove Lemma~\ref{theo:core}.

\subsection{Key properties of core construction}\label{subsec:key-properties}

\renewcommand{\t}{\ell}

For the rest of this subsection fix $1 \le \t \le s$, fix $G \in \G_{\t}$ and put $p = d(G) = 2^{-\t}$.
Our goal in this subsection is to prove the two key properties of $G$ stated in Claim~\ref{claim:property_qr} and Claim~\ref{claim:prop-main} below.
The first key property of $G$ is that, although it is hard to regularize it, $G$ \emph{is} in fact somewhat regular.
Specifically, the measure of regularity is determined by $|\R_1|$.

\begin{claim}[Key Property I]\label{claim:property_qr}
	For every $S \sub \Lside$ and $T \sub \Rside$ with $|S| \ge (4/|\R_1|^{1/6})|\Lside|$ and $|T| \le \frac{1}{256}|\Rside|$,
	we have $e_{G}(S,T) \le \frac{1}{200}p|S||\Rside|$.
	%
	%
	%
\end{claim}
\begin{proof}
	We will prove that $G$ is $(\e)$-regular with $\e = 4/|\R_1|^{1/6}$.
	First, we show that this will complete the proof.
	Note that by assumption~\ref{item:core-minR} of Lemma~\ref{theo:core},
	$\e \le 1/512$.
	Let $S \sub \Lside$ and $T \sub \Rside$ with $|S| \ge \e|\Lside|$ and $|T| \le |\Rside|/512$.
	If $|T| \ge \e|\Rside|$ then $e(S,T) \le (1+\e)p|S||T| \le p|S||\Rside|/200$, and we are done.
	Otherwise, let $T' \sub \Rside$ be an arbitrary superset of $T$ of size $\e|\Rside|$.
	Then $e(S,T) \le e(S,T') \le (1+\e)p|S||T'| = \e(1+\e)p|S||\Rside| \le p|S||\Rside|/200$, and we are done again.
	This completes the proof.
	
	
	We now prove that $G$ is indeed $(\e)$-regular.
	Put $\R_0=\{\Rside\}$ and $\L_0=\{\Lside\}$.	
	For $R,R' \in \R_i$ with $0 \le i \le \t$ we denote by $0 \le r(R,R') \le i$ the largest integer such that $R,R' \sub \h{R}$ for some $\h{R} \in \R_{r(R,R')}$.
	We henceforth abbreviate $\codeg_{\widetilde{G}_i}(R,R')$ by $\codeg_i(R,R')$.
	Recall $\a_j$ from~(\ref{eq:alpha}).
	We prove, by induction on $0 \le i \le \t$, that for every $R,R' \in \R_i$ we have
	\begin{equation}\label{eq:codeg}
	\codeg_i(R,R') \le 2^{r}4^{-i}|\L_i| \prod_{j=r+1}^i (1+\a_j) \quad\text{ where }r=r(R,R') \;.
	\end{equation}
	For the proof we assume $R \neq R'$, as otherwise $r(R,R')=i$ and $\codeg_i(R,R')=\deg_{\widetilde{G}_i}(R)$, hence the required bound $\codeg_i(R,R')=2^{-i}|\L_i|$ follows from~\ref{item:equitable} in Claim~\ref{claim:core-properties}.
	The induction basis $i=0$ is trivially true since $\R_0$ does not contain two distinct members.
	%
	%
	For the induction step, let $R \neq R' \in \R_i$. 
	Let $\h{R},\h{R}' \in \R_{i-1}$ satisfy $R \sub \h{R}$ and $R' \sub \h{R}'$, and let $\h{L} \in \L_{i-1}$ be a common neighbor of $\h{R}$ and $\h{R}'$ in $\widetilde{G}_{i-1}$.
	By~\ref{item:codeg} in Claim~\ref{claim:core-properties} together with the fact that $R \neq R'$, the number of common neighbors in $\widetilde{G}_i$ of $R,R'$ that lie in $\h{L}$ is at most $(1+\a_i)\frac14|\h{L}| = (1+\a_i)\frac14(|\L_i|/|\L_{i-1}|)$.
	Observe that $r(\h{R},\h{R}') = r(R,R')$, again using the fact that $R \neq R'$.
	Let $r=r(R,R')$. It follows that
	\begin{align*}
	\codeg_i(R,R') &\le \codeg_{i-1}(\h{R},\h{R}') \cdot (1+\a_i)\frac14\frac{|\L_i|}{|\L_{i-1}|}\\
	&\le 2^{r}4^{-i+1}|\L_{i-1}| \prod_{j=r+1}^{i-1} (1+\a_j) \cdot (1+\a_i)\frac14\frac{|\L_i|}{|\L_{i-1}|}
	= 2^{r}4^{-i}|\L_i| \prod_{j=r+1}^{i} (1+\a_j)  \;,
	\end{align*}
	where the second inequality uses the induction hypothesis.
	This completes the inductive proof.
	
	Note that Assumption~\ref{item:core-expR} in Lemma~\ref{theo:core} implies that
	\begin{equation}\label{eq:4-power}
	|\R_i| \ge 4^{i-1}|\R_1| \;.
	\end{equation}
	Put $\a=2\sum_{j=1}^{\t} \a_j$.
	By~(\ref{eq:alpha}), (\ref{eq:4-power}) and Assumptions~\ref{item:core-minR} and~\ref{item:core-expL} of Lemma~\ref{theo:core} we have
	\begin{equation}\label{eq:alpha2}
	\sum_{j=1}^{\t} \a_j \le \sum_{j=1}^{\t} \frac{1}{|\L_j|^{1/6}}
	\le \sum_{j=1}^{\t} \frac{1}{2^{|\R_j|/2^{j+13}}}
	\le \sum_{j=1}^{\t} \frac{1}{2^{|\R_1|2^{j-15}}}
	\le \frac{2}{2^{|\R_1|/2^{14}}} \le \frac{2}{|\R_1|} \,\,(\le 1) \;.
	\end{equation}
	For $v,v' \in \Rside$, with $v \in R \in \R_{\t}$ and $v' \in R' \in \R_{\t}$, we write $r(v,v'):=r(R,R')$.
	Since $G$ is a blowup of $\widetilde{G}_{\t}$, we have $\codeg_G(v,v')=\codeg_{\t}(R,R') \cdot |\Lside|/|\L_{\t}|$.
	It follows from the case $i=\t$ of~(\ref{eq:codeg}) (recall $4^{-\t}=p^2$) that
	\begin{equation}\label{eq:codegree2}
	\codeg_G(v,v') \le  2^r p^2|\Lside|\prod_{j=r+1}^{\t} (1+\a_j)
	\le 2^r p^2|\Lside| (1 + \a) \quad\text{ where }r=r(R,R') \;,
	\end{equation}
	using the bound
	$\prod_{j=1}^{\t} (1+\a_j) \le \exp(\sum_{j=1}^{\t} \a_j)
	\le 1 + 2\sum_{j=1}^{\t} \a_j$, as
	$\sum_{j=1}^{\t} \a_j \le 1$
	by~(\ref{eq:alpha2}).
	
	Fix $v \in \Rside$
	and note that~(\ref{eq:codegree2}) implies that
	\begin{equation}\label{eq:codegree3}
	\sum_{v' \in \Rside} \max\{\codeg_G(v,v') - p^2|\Lside|,\,0\}
	\le p^2|\Lside|\sum_{r=0}^{\t} \sum_{\substack{v' \in \Rside\colon\\r(v,v')=r}} \big( (1+\a)2^r - 1 \big) \;.
	\end{equation}
	Observe that the number of vertices $v' \in \Rside$ satisfying $r(v,v') = r$ is
	at most $|\Rside|/|\R_r|$.
	We have
	\begin{equation}\label{eq:codegree4}
	\frac{1}{|\Rside|}\sum_{r=0}^{\t} \sum_{\substack{v' \in \Rside\colon\\r(v,v')=r}} \big( (1+\a)2^r - 1 \big)
	\le \a + \sum_{r=1}^{\t} \frac{2^{r+1}}{|\R_r|}
	\le \a + \frac{8}{|\R_1|}\sum_{r=1}^\infty \frac{2^{r}}{4^{r}}
	\le \a + \frac{8}{|\R_1|}
	\le \frac{12}{|\R_1|} \;,
	\end{equation}
	where the second inequality uses~(\ref{eq:4-power})
	and the third inequality uses~(\ref{eq:alpha2}).
	Summarizing~(\ref{eq:codegree3}) and~(\ref{eq:codegree4}), we have for every $v \in \Rside$ that
	$$\sum_{v' \in \Rside} \max\{\codeg_G(v,v') - p^2|\Lside|,\, 0\}
	\le \frac{12}{|\R_1|} \cdot p^2|\Lside||\Rside| \;.$$
	Therefore, Lemma~\ref{lemma:codegree} implies that $G$ is $(\e)$-regular with $\e=2(12/|\R_1|)^{1/6} \le 4/|\R_1|^{1/6}$.
	As explained in the first paragraph, this completes the proof.
\end{proof}


We will need the following auxiliary property, where we recall that $G$ is any member of $\G_\ell$ and $p=d(G)=2^{-\ell}$.
\begin{claim}\label{claim:property-degree}
	Let $1 \le i \le \t$ and let $G_i \in \G_i$ be the unique graph with $E(G) \sub E(G_i)$.
	If $L \in \L_i$, $R \in \R_{i}$ satisfy $d_{G_i}(L,R)=1$
	then every vertex $u \in L$ satisfies $d_{G}(u,R)=2^i p$.
	%
	
\end{claim}
\begin{proof}
	For each $i \le j \le \t$ let $G_j \in \G_j$ be the unique graph with $E(G) \sub E(G_j)$.
	Note that $E(G) \sub E(G_{\t-1}) \sub \cdots \sub E(G_i)$.
	We prove by induction on $i \le j \le \t$ that
	$d_{G_{j}}(L_j,R)=2^{i-j}$
	%
	for every $L_j \in \L_j$ with $L_j \sub L$.
	We claim that the case $j=\t$ of the induction would complete the proof. Indeed, for every vertex $u \in L$ we have
	$u \in L_\t \in \L_\t$ for some $L_\t \sub L$; hence, recalling that $G$ is a blowup of the graph $\widetilde{G}$ on $(\L_\t,\R_\t)$, we would get
	$$d_G(u,R)=d_G(L_\t,R) = 2^{i-\ell}=2^i p \;,$$
	as needed.
	%
	%
	We now prove our inductive claim.
	The induction basis $j=i$ is trivial since necessarily $L_j=L$.
	For the induction step, let $L_j \in \L_j$ and suppose $L_j \sub L_{j-1} \in \L_{j-1}$.
	By the induction hypothesis,
	$d_{{G}_{j-1}}(L_{j-1},R)=2^{i-j+1}$.
	By~\ref{item:equitable} in Claim~\ref{claim:core-properties}, for each $R_{j-1} \in N_{j-1}(L_{j-1})$ we have $d_{{G}_j}(L_j,R_{j-1})=\frac12$.
	By the construction of $G_j$ from $G_{j-1}$,
	we deduce that $d_{{G}_j}(L_j,R)=2^{i-j}$, which completes the induction step and the proof.
\end{proof}

The second key property of $G$ is as follows.

\begin{claim}[Key Property II]\label{claim:prop-main}
	Let $1 \le i \le \t$.
	Suppose $P \sub \Lside$ satisfies $P \in_{\frac14} \L_{i-1}$ and $P \notin_\g \L_i$.
	Then there exists $\frac{1}{6}2^{-i}|\R_i|$ clusters $R \in \R_i$ satisfying the following:
	\begin{enumerate}
		\item\label{item-key1} $d_G(P,R) \ge \frac{1}{4} 2^i p$,
		\item\label{item-key2} there is $P_1 \sub P$ with $|P_1| \ge \frac18\g|P|$ such that $d_G(P_1,R)=0$.
	\end{enumerate}
	%
	%
	%
	%
	%
	%
	%
	%
\end{claim}
\newcommand{\Nh}{\Gamma}
\begin{proof}
	For each $L' \in \L_i$ let $\lambda_{L'} = |P \cap L'|/|P|$.
	Put $\lambda=(\lambda_{L'})_{L' \in \L_i} \in [0,1]^{|\L_i|}$ and note that $\norm{\lambda}_1 = 1$.
	Let $L \in \L_{i-1}$ be the unique cluster satisfying $P \sub_{1/4} L$.
	We have $\norm{\lambda}_\infty \le 1-\g$ and $\sum_{L'\in\L_i \colon L' \nsubseteq L} \lambda_{L'} \le \frac14$.
	Therefore, by Claim~\ref{claim:1-12} there are at least $\frac{1}{6}2^{-i}|\R_{i}|$ clusters $R' \in \Nf(L)$ for which
	$$\sum_{L' \notin N_{i}(R')} \lambda_{L'} \ge \frac18(1-\norm{\lambda}_{\infty}) \ge \frac18\g \quad\text{ and }\quad
	\sum_{L' \in N_{i}(R'), L' \sub L} \lambda_{L'} \ge \frac12-\sum_{L'\in\L_i \colon L' \nsubseteq L} \lambda_{L'} \ge \frac14 \;.$$
	Fix $R' \in \Nf(L)$ as above.
	To complete the proof it suffices to show that $R'$ satisfies Properties~\ref{item-key1} and~\ref{item-key2} in the statement.
	Denote
	$$P_1(R') = P \cap \bigcup_{L' \notin N_{i}(R')} L' \quad\text{ and }\quad
	P_2(R') = P \cap \bigcup_{L' \in N_{i}(R'), L' \sub L} L' \;.$$	
	Thus, we have
	$$ |P_1(R')| \ge \frac18\g|P| \quad\text{ and }\quad |P_2(R')| \ge \frac14|P| \;.$$	
	Let $G_i \in \G_i$ be the unique graph with $E(G) \sub E(G_i)$.
	We have $e_G(L',R')=0$ for every $L' \notin N_{i}(R')$,  implying
	$$
	e_G(P_1(R'),R')=0 \;,
	$$
	which proves that $R'$ satisfies Property~\ref{item-key2}.
	Next, we claim that $d_{G}(u,R') = 2^i p$ for every $u \in P_2(R')$.
	Indeed, letting $u \in L' \in N_{i}(R')$ 
	(recall the definition of $P_2(R')$), we have $d_{G_i}(L',R')=1$, thus applying Claim~\ref{claim:property-degree} on $L'$ and $R'$ proves our claim.
	We deduce that
	$$
	e_G(P,R') \ge e_G(P_2(R'),R') = |P_2(R')| \cdot 2^i p|R'| \ge \frac14 2^i p|P||R'| \;,
	$$
	which proves that $R'$ satisfies Property~\ref{item-key1}. This completes the proof.
\end{proof}

\subsection{Proof of Lemma~\ref{theo:core}}\label{subsec:LB-core-main}






We are now ready to prove our main result in this section, Lemma~\ref{theo:core}.
For convenience, we restate it below where $G$ is any graph of $\G_{\t}$, for any $1 \le \t \le s$.
We remind the reader that $\langle \d \rangle$-regularity below refers to Definition \ref{def:star-regular}.



\setcounter{lemma}{4}
\begin{lemma}\label{theo:cSRAL}
	Let $\d \le 2^{-20}$.
	Suppose $\P \cup \Q$ is a $\langle \d \rangle$-regular partition of $G$, where $\P$ and $\Q$ are partitions of $\Lside$ and $\Rside$ respectively, and $\Q \prec_{2^{-9}} \R_{t}$ for some $1 \le t \le \t$.
	Then $\P \prec_{\g} \L_{t}$ with
	$\g = \max\{2^{5} \sqrt{\d},\, 32/\sqrt[6]{|\R_1|} \}$.
	%
	%
	%
	%
\end{lemma}

\begin{proof}
	Suppose towards contradiction that $\P \nprec_{\g} \L_{t}$.
	Let $G'$ be a graph satisfying that for every cluster pair $(P,Q) \in \P \times \Q$ and subsets $S \sub P$ and $S' \sub Q$ with $|S| \ge \d|P|$ and $|S'|\ge\d|Q|$ we have $d_{G'}(S,S') \ge \frac12 d_{G'}(P,Q)$.
	We need to show that $G'$ is not $\d$-close to $G$, that is, $|E(G') \Delta E(G)| > \d \cdot e(G)$.
	For subsets $S \sub \Lside$, $T \sub \Rside$ we denote $\Delta(S,T)=|E_{G'}(S,T) \triangle E_{G}(S,T)|$.
	Note that $|E(G') \triangle E(G)| = \sum_{P \in \P} \Delta(P,\Rside)$,
	hence our goal is to prove
	\begin{equation}\label{eq:LB-core-goal}
	\sum_{P \in \P} \Delta(P,\Rside) > \d \cdot e(G) \;.
	\end{equation}

	Put $\L_0 = \{\Lside\}$. 
	For each $1 \le i \le t$ let
	$$\D_i = \Big\{ P \in \P \,\Big\vert\,
	P \in_\g \L_{i-1}\text{ and } P \notin_\g \L_{i} \Big\} \;.$$
	Observe that $\D = \{ P \in \P \,\vert\, P \notin_\g \L_{t} \}$ is the disjoint union $\D = \bigcup_{i=1}^t \D_i$, since $P \in \D$ implies $P \in \D_i$ for a unique $i$ (as $P \in_\g \L_i$ implies $P \in_\g \L_{i-1}$) and since $P \sub_0 \L_0$ for every $P \in \P$.
	We have
	\begin{equation}\label{eq:LB-core-d}
	\sum_{P \in \D} |P| = \sum_{P \in \P \colon P \notin_\g \L_t} |P| > \g |\Lside| \;,
	\end{equation}
	where the inequality uses our assumption $\P \nprec_{\g} \L_{t}$.
	
	Fix $1 \le i \le t$ and put  $c=2^{-9}$.
	Let $\Rside^*_i = \bigcup (Q \cap R)$ where the union is over
	all $R \in \R_i$ and $Q \in \Q$ satisfying $Q \sub_{c} R$. We have
	\begin{equation}\label{eq:LB-core-R1}
	|\Rside \sm \Rside^*_i| = \sum_{R \in \R_i} \sum_{\substack{Q \in \Q \colon \\ Q \nsubseteq_{c} R}} |Q \cap R|
	\le \sum_{\substack{Q \in \Q \colon \\ Q \notin_{c} \R_i}} |Q|
	+  \sum_{\substack{Q \in \Q \colon \\ Q \in_{c} \R_i}} c|Q|
	\le 2^{-8}|\Rside| \;,
	\end{equation}
	where the last inequality bounds the first summand using the fact that $\Q \prec_c \R_i$, which follows from the statement's assumption $\Q \prec_c \R_t$ together with the fact that $\R_t \prec \R_i$.
	We also record the fact that, again since $\R_i \prec \R_t$, we have
	\begin{equation}\label{eq:LB-core-Rside*}
	\Rside^*_t \sub \Rside^*_i \;.
	\end{equation}
	
	Fix $P \in \D_i$ and put $p=d(G)$.
	Our goal is to prove the lower bound on $\Delta(P,\Rside)$ stated in~(\ref{eq:LB-Delta}) below.	
	Note that $P \in_\frac14 \L_{i-1}$ (since $\g \le \frac14$ by the statement's assumptions) and $P \notin_\g \L_{i}$.
	It follows by applying Claim~\ref{claim:prop-main} on
	$P$ that there exist $\frac{1}{6}2^{-i}|\R_i|$ clusters $R \in \R_i$ where for every such $R$ there is a subset $P_1=P_1(R) \sub P$ such that the following hold:
	\begin{enumerate}
		\item $|P_1| \ge \frac{1}{8}\g|P| \ge \d|P|$, 
		\item $d_G(P_1,R) = 0$,
		\item $d_G(P,R) \ge \frac14 2^i p$.
	\end{enumerate}
	Fix a cluster $R \in \R_i$ as above.
	Put $\g' = 2^{-5}\g$.
	Our goal toward proving~(\ref{eq:LB-Delta}) is to prove the lower bound on $\Delta(P,R)$ stated in~(\ref{eq:LB-Delta-R}) below.
	First, let $Q \in \Q$ and assume $Q \sub_c R$.
	Then
	\begin{align}
	\begin{split}\label{eq:LB-Q}
	e_{G'}(P_1,Q \cap R)
	&= |P_1|\cdot |Q \cap R|\cdot d_{G'}(P_1,Q \cap R) \\
	&\ge \frac{1}{8}\g|P|\cdot \frac12|Q|\cdot \frac12d_{G'}(P,Q)
	= \g' \cdot e_{G'}(P,Q)
	\ge \g' \cdot e_{G'}(P,Q \cap R) \;,
	\end{split}
	\end{align}
	where the first inequality follows from the assumption on the regularity of $\P \cup \Q$ at the beginning of the proof, together with~$(i)$ and our assumption $|Q \cap R| \ge (1-c)|Q| \ge \frac12|Q|$.
	Let $R^* = R \cap \Rside^*_i$. We have by definition that $R^*$ decomposes into sets $Q \cap R$ such that $Q \in \Q$ satisfies $Q \sub_c R$.
	Therefore,
	\begin{align}
	\begin{split}\label{eq:LB-mod1}
	\Delta(P_1,R^*) &= e_{G'}(P_1,R^*)
	= \sum_{\substack{Q \in \Q\colon\\Q \sub_c R}} e_{G'}(P_1,Q \cap R) \\
	&\ge \g'\sum_{\substack{Q \in \Q\colon\\Q \sub_c R}} e_{G'}(P,Q \cap R)
	= \g' \cdot e_{G'}(P,R^*) \;,
	\end{split}
	\end{align}
	where the first equality follows from fact that $e_G(P_1,R^*)=0$ by~$(ii)$,
	and the inequality follows from~(\ref{eq:LB-Q}).
	%
	%
	%
	Now, to bound the right hand side of~(\ref{eq:LB-mod1}) we do the following, where we henceforth write $P_2 = P \sm P_1$;
	\begin{align}
	\begin{split}\label{eq:LB-mod2}
	e_{G'}(P,R^*) &\ge e_{G'}(P_2,R^*)
	\ge e_G(P_2,R^*) - \Delta(P_2,R^*)\\
	&= \big(e_G(P_2,R) - e_G(P_2,R \sm \Rside^*_i)\big) - \Delta(P_2,R^*)\\
	&\ge \frac14 2^i p|P||R|-e_G(P,R \sm \Rside^*_i) - \Delta(P_2,R) \;,
	\end{split}
	\end{align}
	where the second inequality uses the bound
	$\Delta(S,T) \ge e_{G}(S,T)-e_{G'}(S,T)$,
	and the last inequality bounds the first term using~$(ii)$ (implying $e_G(P_2,R)=e_G(P,R)$) and $(iii)$.
	Since $\Delta(P,R)=\Delta(P_1,R)+\Delta(P_2,R)$,
	the combination of~(\ref{eq:LB-mod1}) and~(\ref{eq:LB-mod2})
	implies the lower bound
	\begin{equation}\label{eq:LB-Delta-R}
	\Delta(P,R)
	\ge \g'\Big(\frac{1}{4} 2^i p|P||R| - e_G(P,R \sm \Rside^*_i) \Big) \;.
	\end{equation}
	Summarizing, for every $P \in \D_i$ we have
	\begin{equation}\label{eq:LB-Delta}
	\Delta(P,\Rside)
	\ge \sum_{R} \Delta(P,R)
	\ge \g'\Big(\frac{1}{24} p|P||\Rside| - e_G(P,\,\Rside \sm \Rside^*_t) \Big) \;,
	\end{equation}
	with the sum over the $\frac16 2^{-i}|\R_i|$ clusters $R \in \R_i$ as above,
	where the second inequality follows from~(\ref{eq:LB-Delta-R}) using the fact that $|R|=|\Rside|/|\R_i|$ for every $R \in \R_i$, as well as the fact that $\Rside \sm \Rside^*_i \sub \Rside \sm \Rside^*_t$ by~(\ref{eq:LB-core-Rside*}).
	
	Let $\Df = \bigcup_{P \in \D} P$ and recall that, by~(\ref{eq:LB-core-d}), $|\Df| \ge \g|\Lside|$.
	Summing over all $P \in \D$, we finally get
	\begin{align*}
	\sum_{P \in \D} \Delta(P,\Rside) &=
	\sum_{i=1}^t \sum_{P \in \D_i} \Delta(P,\Rside)
	\ge \g'\Big(\frac{1}{24}p|\Df||\Rside| - e_G(\Df,\Rside \sm \Rside^*_t) \Big)\\
	&\ge \g'\Big( \frac{1}{24}p|\Df||\Rside| - \frac{1}{200}p|\Df||\Rside| \Big)
	> 2^{-5}\g' \cdot p|\Df||\Rside| \ge 2^{-10}\g^2 \cdot p|\Lside||\Rside| \;,
	\end{align*}
	where the first inequality uses~(\ref{eq:LB-Delta}),
	the second inequality uses Claim~\ref{claim:property_qr} together with~(\ref{eq:LB-core-R1}) and the fact that $|\Df| \ge \g|\Lside| \ge 4|\R_1|^{-1/6}|\Lside|$
	by the choice of $\g$, and the last inequality again uses $|\Df| \ge \g|\Lside|$.
	
	We therefore deduce that
	$\sum_{P \in \D} \Delta(P,\Rside) > \d\cdot e(G)$,
	again by the choice of $\g$.
	This proves~(\ref{eq:LB-core-goal}), thus completing the proof.
\end{proof}

\section{Ackermann-type Lower Bounds for the R\"odl-Schacht Regularity Lemma}\label{sec:coro}


\renewcommand{\K}{\mathcal{K}}

The purpose of this section is to apply Theorem~\ref{theo:main} in order to prove Corollary~\ref{coro:RS-LB}, giving level-$k$ Ackermann-type lower bounds for the $k$-graph regularity lemma of R\"odl and Schacht~\cite{RodlSc07}.
We remind the reader that in~\cite{MS3} one can find an analogous section which deduces wowzer-type lower bound for the $3$-graph regularity lemmas of Frankl and R\"odl \cite{FrankRo02} and of Gowers~\cite{Gowers06}.
We begin with the required definitions.
The definitions we state here are essentially equivalent to (though shorter than) those in~\cite{RodlSc07}.
We will rely on the definitions in Subsection~\ref{subsec:preliminaries}, and in particular, the definition of a $k$-partition.

For a $k$-graph $H$, the \emph{density} of a $(k-1)$-graph $S$ in $H$ is
$$d_H(S) = \frac{|H \cap \K(S)|}{|\K(S)|} \;,$$
%
%
where $d_H(S)=0$ if $|\K(S)|=0$.
The notion of $\e$-regularity for $k$-graphs is defined as follows.

\begin{definition}[$\e$-regular $k$-graph]\label{def:e-reg}
	A $k$-partite $k$-graph $H$ is \emph{$(\e,d)$-regular}---or simply \emph{$\e$-regular}---in a $k$-polyad $P$ with $H \sub \K(P)$
	if for every $S \sub P$ with $|\K(S)| \ge \e|\K(P)|$ we have $d_H(S) = d \pm \e$.
\end{definition}

A \emph{partition} of a $k$-graph $H$ is simply a $(k-1)$-partition on $V(H)$.

\begin{definition}[$\e$-regular partition]\label{def:e-reg-partition}
	A partition $\P$ of a $k$-graph $H$ is \emph{$\e$-regular} if
	$\sum_P |\K(P)| \le \e|V(H)|^k$ where the sum is over all $k$-polyads $P$ of $\P$ for which
	$H \cap \K(P)$ is not $\e$-regular in $P$.	
\end{definition}

Henceforth, an \emph{$(r,a_1,\ldots,a_r)$-partition} is simply an $r$-partition $\P$ (recall Definition~\ref{def:r-partition}) where  $|\P^{(1)}|=a_1$ and for every $2 \le i \le r$,
$\P^{(i)}$ subdivides each $K \in \K_i(\P)$ into $a_i$ parts.

\begin{definition}[$f$-equitable partition]\label{def:r-equitable}
	Let $f\colon[0,1]\to[0,1]$.
	An $(r,a_1,\ldots,a_r)$-partition $\P$ is \emph{$f$-equitable} if $\P^{(1)}$ is equitable and for every $2 \le i \le r$, every $i$-graph $F \in \P^{(i)}$ is $(\e,1/a_i)$-regular in $\under(F)$, where $\e=f(d_0)$ and $d_0=\min\{1/a_2,\ldots,1/a_{r}\}$.
\end{definition}

\subsection{The lower bound}

The $k$-graph regularity of R\"odl and Schacht~\cite{RodlSc07} states, roughly, that for every $\e>0$ and every function $f\colon\N\to(0,1]$, every $k$-graph has an $\e$-regular $f$-equitable equipartition $\P$ where $|\P|$ (the total number of elements in the $(k-1)$-partition $\P$) is bounded by a level-$k$ Ackermann-type function.
In fact, R\"odl-Schacht's $k$-graph regularity lemma (Theorem~2.3 in~\cite{RodlSc07}) uses a considerably stronger notion of regularity of a partition that involves an additional function $r$ which we shall not discuss here (this stronger notion was crucial in~\cite{RodlSc07} for allowing them to prove a counting lemma).
Our lower bound below applies even to the weaker notion stated above, which corresponds to taking $r \equiv 1$.

The proof of  Corollary~\ref{coro:RS-LB}
will follow quite easily from Theorem~\ref{theo:main} together with Claim~\ref{claim:k-reduction} below.
Claim~\ref{claim:k-reduction} basically shows that a $\langle \d \rangle$-regularity ``analogue'' of R\"odl and Schacht's notion of regularity implies graph $\langle \d \rangle$-regularity.
%
The proof of Claim~\ref{claim:k-reduction} is deferred to the Appendix~\ref{sec:RS-appendix}.
Henceforth we say that a graph partition is \emph{perfectly $\langle \d \rangle$-regular} if all pairs of distinct clusters are $\langle \d \rangle$-regular without modifying any of the graph's edges.
\begin{claim}\label{claim:k-reduction}
	Let $H$ be a $k$-partite $k$-graph on vertex classes $(\Vside^1,\ldots,\Vside^k)$, and
	let $\P$ be an $f$-equitable partition of $H$ with $\P^{(1)} \prec \{\Vside^1,\ldots,\Vside^k\}$,
	$f(x) = \d^4(x/2)^{2^{k+3}}$ and $|V(H)| \ge n_0(\d,|\P|)$.
	Suppose that for each $k$-polyad $P$ of $\P$, every $S \sub P$ with $|\K(S)| \ge \d|\K(P)|$ has $d_{H}(S) \ge \frac23 d_{H}(P)$.
	Then $E_k(\P) \cup V_k(\P)$ is a perfectly $\langle 2\sqrt{\d} \rangle$-regular 
	partition of $G_H^k$.
\end{claim}

Although quite technical, the rough idea behind the proof of Claim~\ref{claim:k-reduction} is rather simple.
In order to argue about the bipartite graphs in the statement's partition of $G_H^k$, one should consider ``semi-complete'' $k$-polyads, that is, $k$-polyads of the form $F \circ V$ (see Subsection~\ref{subsec:k-partitions-proofs}).
Using Claim~\ref{claim:decomposition}, one partitions such a $k$-polyad into $k$-polyads of $\P$, for which the statement's regularity assumption applies. The so-called dense counting lemma of~\cite{RodlSc07-B} implies that these $k$-polyads span approximately the expected number of $k$-cliques. By going over all the aforementioned $k$-polyads of $\P$, this can be used to deduce the regularity of the corresponding bipartite graph, $G_H^k[F,V]$, as needed.


We now formally restate and prove Corollary~\ref{coro:RS-LB}.
%
We mention that, as will be immediate from the proof, our lower bound not only applies to the hypergraph regularity lemma of R\"odl and Schacht but also to the hypergraph regular approximation lemma~\cite{RodlSc07}.

\begin{theo}[Lower bound for the R\"odl-Schacht $k$-graph regularity lemma]\label{theo:RS-LB}
	Let $s \ge k \ge 2$ and
	put $c = 2^{-32^k}$.
	For every $s \in \N$ there exists a $k$-partite $k$-graph $H$ of density $p=2^{-s-k}$, and a partition $\V_0$ of $V(H)$ with $|\V_0| \le k 2^{200}$,
	such that if $\P$ is an $\e$-regular $f$-equitable partition of $H$ 	
	with
	$\e \le c p$,
	$f(x) \le c^4(x/2)^{2^{k+3}}$, $|V(H)| \ge n_0(k,|\P|)$
	and $\P^{(1)} \prec \V_0$, then $|\P^{(1)}| \ge \Ack_k(s)$.
\end{theo}

\begin{remark}
	One can easily remove the assumption $\P^{(1)} \prec \V_0$ by taking the common refinement of $\P^{(1)}$ with $\V_0$ (and adjusting $\P$ appropriately). Since $|\V_0|=O(k)$, this has only a minor effect on the parameters of $\P$ and thus
	one gets essentially the same lower bound. We omit the details of this routine transformation.
	%
\end{remark}


\begin{proof}[Proof of Theorem~\ref{theo:RS-LB}]
	Put $\a_k = 2^{-16^k}$ (recall $c = 2^{-32^k}$).
	The bound $|\P^{(1)}| \ge \Ack_k(s)$ would follow from Theorem~\ref{theo:main} if we show that $H$ is $\langle \a_k \rangle$-regular relative to the $(k-1)$-partition $\P$.
	Henceforth put $\d=(\a_k/2)^2$. We will later use the inequalities
	\begin{equation}\label{eq:RS-LB-ineq}
	c \le \frac{\a_k^k}{7k^k} \le \d \;.
	\end{equation}
	First we claim that $\P$ is $\langle \a_k \rangle$-good (recall Definition~\ref{def:k-good}).
	Let $2 \le r \le k-1$, let $F \in \P^{(r)}$ be an $r$-partite $r$-graph on $(V_1,\ldots,V_r)$, and denote by $P=\under(F)$ the $r$-polyad underlying $F$.	
	We will show that the bipartite graph $G_{F,P}^r$ is $\langle \a_k \rangle$-regular, and since an analogous argument will hold for $G_{F,P}^i$ for every $1 \le i \le r$, this would prove our claim.
	Recalling Definition~\ref{def:aux}, we have that $G_{F,P}^r = G_F^r[E,V_r]$ with $E:=P[V_1,\ldots,V_{r-1}] \in E_r(\P)$.
	Now, suppose $\P$ is a $(k-1,a_1,\ldots,a_{k-1})$-partition, put $d_r = 1/a_r$ for each $2 \le r \le k-1$, and put $d_0=\min\{1/a_2,\ldots,1/a_{k-1}\}$.
	Recalling Definition~\ref{def:r-equitable}, since $\P$ is $f$-equitable we have that $F$ is $(f(d_0),d_r)$-regular in $P$. Thus, recalling Definition~\ref{def:e-reg}, for every $S \sub P$ with $|K(S)| \ge \d|K(P)|$ (note $\d \ge c \ge f(d_0)$ using~(\ref{eq:RS-LB-ineq})) we have $d_F(S) \ge d_r-f(d_0) \ge d_F(P) - 2f(d_0) \ge \frac23 d_F(P)$.
	Let the $(r-1)$-partition $\P'$ be obtained by restricting $\P$ to $V_1 \cup\cdots \cup V_r$ (so in particular $V_i(\P')=\{V_i\}$), and note that $P$ is an $r$-polyad of $\P'$.
	Observe that $d_F(S) \ge \frac23 d_F(P)$ trivially holds for any $r$-polyad $P'$ of $\P'$ other than $P$ as well, since $d_F(P')=0$ as $F \sub \K(P)$.
	Apply Claim~\ref{claim:k-reduction}, with (the almost trivial choice of) the $r$-partite $r$-graph $F$ and the $f$-equitable $(r-1)$-partition $\P'$ of $F$, to deduce that $E_{r}(\P') \cup \{V_r\}$
	is a perfectly $\langle\a_k\rangle$-regular (i.e., $\langle 2\sqrt{\d} \rangle$-regular) partition of $G_{F}^r$.
	Since $E \in E_{r}(\P')$, this in particular implies that $G_F^r[E,V_r]$ is $\langle\a_k\rangle$-regular, which proves our claim as explained above.
	
	
	%
	
	It remains to show that $H$ is $\langle \a_k \rangle$-regular relative to the  $\langle \a_k \rangle$-good $(k-1)$-partition $\P$.
	Let $H'$ be obtained from $H$ by removing all its ($k$-)edges underlain by $k$-polyads of $\P$ such that either $d_H(P) \le 6\e$ or the $k$-graph $H \cap \K(P)$ is not $\e$-regular in $P$.
	By Definition~\ref{def:e-reg-partition}, the number of edges removed from $H$ to obtain $H'$ is at most
	$$\e|V(H)|^k + 6\e|V(H)|^k \le 7\cdot c p |V(H)|^k
	\le (\a_k p/k^k)|V(H)|^k = \a_k\cdot e(H) \;,$$
	where the inequalities use the statement's assumption on $\e,c$ and~(\ref{eq:RS-LB-ineq}), and the equality uses the fact that all $k$ vertex classes of $H$ are of the same size (see Remark~\ref{remark:main}).
	Thus, in $H'$, every non-empty $k$-polyad of $\P$ is $\e$-regular and of density at least $6\e$.
	Again by Definition~\ref{def:e-reg-partition}, for every $k$-polyad $P$ of $\P$ and every $S \sub P$ with $|\K(S)| \ge \d |\K(P)|$ ($\ge \e |\K(P)|$
	by~(\ref{eq:RS-LB-ineq})) we have $d_H(S) \ge d_H(P)-2\e \ge \frac23 d_H(P)$.
	Apply Claim~\ref{claim:k-reduction}, this time with $H$ and $\P$.
	It follows that
	$E_{k}(\P) \cup V_{k}(\P)$
	is an $\langle \a_k \rangle$-regular 
	partition of $G_H^k$. An analogous argument applies for $G_H^i$ for every $1 \le i \le k$, thus completing the proof.
\end{proof}

\section{$\langle \d \rangle$-regularity does not suffice for triangle counting}\label{sec:example}

As observed following the statement of Theorem~\ref{theo:main}, our lower bound for $\langle \d \rangle$-regularity holds even when $\d$ is a fixed\footnote{As noted following Theorem~\ref{theo:main}, the fact that Theorem~\ref{theo:main} holds even with a {\em fixed} $\d=\d(k)$ (which is allowed to be much larger than the edge density $p$) is crucial for our inductive proof strategy.} constant and the edge density is small. 
As Proposition~\ref{claim:example} below states, this setting\footnote{On the other hand, if $\d$ is small compared to the density then a tripartite $\langle \d \rangle$-regular graph does indeed have many triangles, using a standard proof of the counting lemma.} is not strong enough even for counting triangles in graphs.
Namely, we construct for every fixed $\delta>0$ and small enough $p$, a graph of density $p$ which is $\langle \d \rangle$-regular yet does not
even contain a single triangle.
The precise statement is the following.

\begin{prop}\label{claim:example}
For every $0 < p \le 10^{-3}\d^5$ and large enough $n$ there is a $n$-vertex tripartite graph of density at least $p$, whose every pair of classes span a $\langle \d \rangle$-regular graph, and yet is triangle free.
\end{prop}
We use the following well-known lemma, where we denote by $\norm{v}_1$ the  $\ell_1$-norm of a vector $v$ (for a proof see, e.g., Lemma 4.3 in~\cite{KaliSh13}).
\begin{lemma}\label{lemma:convex}
	Every vector $x \in [0,1]^n$ with $\norm{x}_1 \in \N$ is a convex combination of binary vectors $y \in \{0,1\}^n$ each with $\norm{y}_1=\norm{x}_1$.
%
%
%
	%
\end{lemma}

We will also apply the following version of the Chernoff bound.
\begin{lemma}[Multiplicative Chernoff bound]\label{lemma:Chernoff}
	Let $X_1,\ldots,X_n$ be mutually independent Bernoulli random variables,
	and put $X=\sum_{i=1}^n X_i$, $\mu = \Ex[X]$.
	For every $\d \in [0,1]$ we have
	$$\Pr(X \neq (1 \pm \d)\mu) \le 2\exp\Big(-\frac13\d^2\mu\Big) \;.$$
\end{lemma}

\begin{proof}[Proof of Proposition~\ref{claim:example}]
	We will in fact construct a graph satisfying an even somewhat stronger property than $\langle \d \rangle$-regularity (namely, the constant $\frac12$ is Definition~\ref{def:star-regular} will be replaced by $1-\d$).
	Consider a random tripartite graph on vertex classes $(V_1,V_2,V_3)$, each of size $k$, obtained by independently retaining each edge with probability 
	$q:=3p$ $(\le 1)$, where $k$ is any integer satisfying
	\begin{equation}\label{eq:example-k}
	64\d^{-2}q^{-1} \le k
	\le \frac14\d^3 q^{-2} \;.
	\end{equation}
	Note that $k$ is well defined in~(\ref{eq:example-k}) by the statement's bound on $p$.
	%
	%
	Let $X$ be the random variable counting the triangles in the graph. One can easily check that
	$$\Exp[X] = k^3 q^3 \quad\text{ and }\quad
	\Var[X] \le k^3 q^3 + 3\binom{k}{2}k^2 q^5 \;.$$
	Chebyshev's inequality thus implies that
	$$\Pr[X \ge \d^3k^2 q] \le \frac{\Var[X]}{(\d^3k^2q - \Exp[X])^2}
	\le \frac{k^3q^3(1+\frac32 k q^2)}{k^4q^2(\d^3 - kq^2)^2}
	\le \frac{q}{k} \cdot 8\d^{-6}
	\le \frac18 q^2 \d^{-4}
	= \frac98 p^2 \d^{-4}
	< \frac12 \;,$$
	where the third inequality uses the upper bound $kq^2 \le \frac14 \d^3$ $(\le \frac14)$ from~(\ref{eq:example-k}), the fourth inequality uses the lower bound from~(\ref{eq:example-k}), and the last inequality uses the statement's bound on $p$.
	Next, by using Lemma~\ref{lemma:Chernoff}
	together with the union bound we deduce that 
	\begin{equation}\label{eq:example-property}
	\forall 1 \le a<b \le 3\,\,\,
	\forall S \sub V_a,\, T \sub V_b \text{ with } |S| \ge \d|V_a|,\, |T| \ge \d|V_b| \,\,\colon\,\,
	d(S,T) = \big(1\pm\frac13\d\big)q
	\end{equation}
	except with probability at most
	$$3 \cdot 2^{2k} \cdot 2\exp\Big(-\frac{1}{27}\d^2 qk^2\Big) \le \frac12 \;,$$
	where the inequality uses the lower bound $kq \ge 64\d^{-2}$ from~(\ref{eq:example-k}).
	We deduce from all of the above that there exists a tripartite graph that has $k$ vertices in each vertex class, at most $\d^3 k^2 q$ triangles and satisfies~(\ref{eq:example-property}).
	By removing an edge from each triangle, one-third of them from each of the three pairs of vertex classes, we obtain a triangle-free graph $G_0$ such that for every $a<b$ and subsets $S \sub V_a$, $T \sub V_b$
	with $|S|\ge\d|V_a|$, $|T| \ge \d|V_b|$ we have
	$$e(S,T) \ge \big(1-\frac13\d\big)q|S||T|-\frac13 \d^3k^2q \ge \big(1-\frac23\d\big)q|S||T|
	\ge \frac{1-\frac23\d}{1+\frac13\d}d(V_a,V_b)|S||T|
	\ge (1-\d)d(V_a,V_b)|S||T|$$
	where the first and third inequalities follows from the lower and upper bound in~(\ref{eq:example-property}), respectively.
	In particular, we deduce from the inequality $e(S,T) \ge (1-\frac23\d)q|S||T|$ above that $d(V_a,V_b) \ge \frac13 q=p$.
	Summarizing, $G_0$ is triangle free, has density at least $p$ and satisfies 
	\begin{equation}\label{eq:example-property2}
	\forall a<b\,\,
	\forall S \sub V_a,\, T \sub V_b \text{ with } |S| \ge \d|V_a|,\, |T| \ge \d|V_b| \,\,\colon\,\, d(S,T) \ge (1-\d)d(V_a,V_b) \;.
	\end{equation}

	To obtain from $G_0$ a graph on a large enough number of vertices we simply take a blow-up, replacing each vertex $v$ of $G_0$ by a set $G(v)$ of $m$ new vertices and every edge by a complete bipartite graph, for $m \in \N$ suitably large. The resulting tripartite graph $G$ is clearly triangle free and of density $d(G_0) \ge p$.
	It thus remains to prove that any two vertex classes of $G$ span a bipartite graph satisfying the desired regularity property.
	Let $a<b$ and let $\overline{V_a},\overline{V_b}$ denote the vertex classes of $G$ corresponding to $V_a,V_b$.
	Note that $|V_a|=|V_b|=k$ and $|\overline{V_a}|=|\overline{V_b}|=mk$.
	Let $S \sub \overline{V_a}$, $T \sub \overline{V_b}$ with $|S| = \d|\overline{V_a}|$ and $|T| = \d|\overline{V_b}|$.
	To complete the proof it suffices to show that 
	\begin{equation}\label{eq:example-goal}
	d(S,T) \ge (1-\d)d(\overline{V_a},\overline{V_b}) \;.
	\end{equation}
	Consider the two vectors $s,t \in [0,1]^k$ defined as follows:
	$$s=(|S \cap G(u)|/m)_{u \in V_a} \quad\text{ and }\quad
	t=(|T \cap G(v)|/m)_{v \in V_b} \;.$$
	Note that $\norm{s}_1 = |S|/m = \d k$ and $\norm{t}_1 = |T|/m = \d k$.
	Assume without loss of generality that $\d k \in \N$.
	By Lemma~\ref{lemma:convex} applied on $s$,
	$$s = \sum_i \a_i s_i \quad\text{ with }\quad s_i \in \{0,1\}^k,\, \norm{s_i}_1 = |S|/m \text{ and } \a_i \ge 0,\, \sum_i \a_i = 1 \;.$$
	By Lemma~\ref{lemma:convex} applied on $t$,
	$$t = \sum_j \b_j t_j \quad\text{ with }\quad t_j \in \{0,1\}^k,\, \norm{t_j}_1 = |T|/m \text{ and } \b_j \ge 0,\, \sum_j \b_j = 1 \;.$$		
	%
	%
	Let $A$ denote the $k \times k$ bi-adjacency matrix of $G_0[V_a,V_b]$.
	Observe that $e_G(S,T) = (ms)^T A (mt)$.
	Moreover, observe that for every $i,j$ we have that $s_i^T A t_j$ is the number of edges of $G_0$ between the subsets of $V_a,V_b$ corresponding to $s_i,t_j$, respectively. Note that these subsets are of size $\norm{s_i}_1,\norm{t_j}_1$, respectively, which are both at least $\d k$.
	Thus, by~(\ref{eq:example-property2}),	
	$$s_i^T A t_j \ge (1-\d)d(V_a,V_b)\norm{s_i}_1\norm{t_j}_1
	= (1-\d)d(\overline{V_a},\overline{V_b})|S||T|/m^2 \;.$$
%
	We deduce that
	\begin{align*}
	e(S,T) &= m^2 \cdot s^T A t
	= m^2\Big(\sum_i \a_i s_i \Big)^T A \Big(\sum_j \b_j t_j\Big)
	= m^2\sum_{i,j} \a_i\b_j s_i^T A t_j\\
	&\ge \Big(\sum_i \a_i \Big)\Big(\sum_j \b_j \Big)(1-\d) d(\overline{V_a},\overline{V_b})|S||T| = (1-\d)d(\overline{V_a},\overline{V_b})|S||T|\;.
	\end{align*}
%
	This gives~(\ref{eq:example-goal}) and thus completes the proof.
\end{proof}

\bigskip

\noindent \textbf{Acknowledgment:} We are grateful to an anonymous referee for a careful reading of the paper.

\appendix

\section{Proof of Claim~\ref{claim:k-reduction}}\label{sec:RS-appendix}



\renewcommand{\K}{\mathcal{K}}

\subsection{Basic facts}

In order to prove Claim~\ref{claim:k-reduction}
we will need several auxiliary results and definitions.
We begin with the notion of complexes.
Henceforth, the \emph{rank} of a (not necessarily uniform)
hypergraph $P$ is $\max_{e \in P} |e|$.
For $r \ge 2$ we denote 
$P^{(r)} = \big\{e \in P \,\big\vert\, |e|=r\big\}$ and $P^{(1)}=V(P)$.

\begin{definition}[complex]
	A \emph{$k$-complex} $(k \ge 2)$ is a $k$-partite hypergraph $P$ of rank $k-1$ where $P^{(r)} \sub \K(P^{(r-1)})$ for every $2 \le r \le k-1$.
\end{definition}


\begin{definition}[$f$-regular complex]\label{def:complex}		
	Let $f \colon [0,1]\to[0,1]$. A $k$-complex $P$ on vertex classes $(V_1,\ldots,V_k)$ is \emph{$(f,d_2,\ldots,d_{k-1})$-regular}, or simply \emph{$f$-regular}, if
	for every $2 \le r \le k-1$ and every $r$ vertex classes $V_{i_1},\ldots,V_{i_r}$ we have that $P^{(r)}[V_{i_1},\ldots,V_{i_r}]$ is $(\e,d_r)$-regular in $P^{(r-1)}[V_{i_1},\ldots,V_{i_r}]$,
	where $\e=f(d_0)$ and $d_0=\min\{d_2,\ldots,d_{k-1}\}$.
	%
\end{definition}

Note that by using the notion of complexes one can equivalently define an $f$-equitable partition (recall Definition~\ref{def:r-equitable}) as follows;
	an $(r,a_1,\ldots,a_r)$-partition $\P$ is $f$-equitable if $\P^{(1)}$ is equitable and, if $r \ge 2$, every $r$-complex
	of $\P$
	is $(f,1/a_2,\ldots,1/a_r)$-regular.

We now state the \emph{dense counting lemma} of~\cite{RodlSc07-B} specialized to
complexes.
%
We henceforth fix the following notation for $k \ge 3$, $\g > 0$;
\begin{equation}\label{eq:DCL}
F_{k,\g}(x) := \frac{\g^3}{12}\Big(\frac{x}{2}\Big)^{2^{k+1}} \;.
\end{equation}
The statement we use below follows from combining Theorem~10 and Corollary~14 in~\cite{RodlSc07-B}, and generalized to the case where the vertex classes are not necessarily of the same size.
For a $k$-polyad $F$ and an edge $e \in F$, we denote the set of $k$-cliques in $F$ containing $e$ by $\K(F,e)=\{e' \in \K(F) \,\vert\, e \sub e'\}$.
For a $k$-complex $P$ we abbreviate $\K(P):=\K(P^{(k-1)})$.
\begin{fact}[Dense counting lemma for $k$-complexes]\label{fact:counting}
	Let $\g > 0$ and let $P$ be a $k$-complex $(k \ge 3)$
	that is $(F_{k,\g},d_2,\ldots,d_{k-1})$-regular with $n_i \ge n_0(\g,d_2,\ldots,d_{k-1})$ vertices in the $\ith$ vertex class.
	Then	
	$$|\K(P)| = (1 \pm \g)\prod_{i=2}^{k-1} d_i^{\binom{k}{i}} \cdot \prod_{i=1}^k n_i \;.$$
	Moreover,\footnote{In~\cite{RodlSc07-B}, the statement of the `moreover' part (Corollary 2.3, dense extension lemma) allows for $\g|P^{(k-1)}|$ exceptional edges in $P^{(k-1)}$ rather than only in $P_k$, which is nevertheless essentially equivalent to our statement. Furthermore, they allow for counting not only $k$-cliques, in which case they do not need all $P_i$ to be regular.} let $P^{(k-1)}=(P_1,\ldots,P_k)$. We have for all edges $e \in P_k$ but at most $\g|P_k|$ that
	$$|\K(P,e)| = (1 \pm \g)\prod_{i=2}^{k-1} d_i^{\binom{k-1}{i-1}} \cdot n_{k} \;.
	\footnote{To obtain the bound $F_{k,\g}$ from the proof of Corollary~2.3 in~\cite{RodlSc07-B} (with $h=k-1$ and $\ell=k$), one can verify that:
		\begin{itemize}
		\item $\e(\mathbf{\F},\g,d_0)$ in Theorem~2.2 can be bounded by $\g(d_0/2)^{|\F^{(h)}|}$, and so $\e(K_{k}^{(k-1)},\g,d_0) \le \g(d_0/2)^{2^k}$,
		\item $\b$ in Fact~2.4 can be bounded by $\g^3/4$,
		\item $\e_{GDCL}(\D(\F^{(h)},f),\frac{\b}{3},d_0)$ in the proof of Corollary~2.3 can be bounded by $\frac{\b}{3}(d_0/2)^{2^{k+1}}$, using the first item and the fact that $\D(\F^{(h)},f)$ has at most $2k-(k-1)=k+1$ vertices.
		\end{itemize}}$$
\end{fact}


We will also need a \emph{slicing lemma} for complexes.

\begin{lemma}[Slicing lemma for complexes]\label{lemma:k-slice}
	Let $P$ be a $k$-complex $(k\ge 3)$ on vertex classes $(V_1,\ldots,V_k)$ and let $V_k' \sub V_k$ with $|V_k'| \ge \d|V_k|$.
	If $P$ is $(f,\,d_2,\ldots,d_{k-1})$-regular with
	$f(x) \le \frac{\d}{2} F_{k-1,\frac14}(x)$
	and $|V(P)| \ge n_0(d_2,\ldots,d_{k-1})$
	then the induced $k$-complex $Q=P[V_1,\ldots,V_{k-1},V_k']$ is $(f^*,d_2,\ldots,d_{k-1})$-regular with $f^*=\frac{2}{\d} \cdot f$.
\end{lemma}
For the proof we will need the notation $P^{(\le i)} = \{e \in P \,\vert\, |e| \le i\}$ where $P$ is any hypergraph.
\begin{proof}
	We proceed by induction on $k$. We begin with the induction basis $k=3$.
	Let $P=(P_1,P_2,P_3)$ be an $(f,d_2)$-regular $3$-complex on vertex classes $(V_1,V_2,V_3)$, meaning that each bipartite graph $P_i$ (which is obtained from $P$ by removing $V_i$ and its adjacent edges)
	is $(\e,d)$-regular with $d=d_2$ and $\e=f(d)$.
	Put $Q=(Q_1,Q_2,Q_3)$.
	We will show that the bipartite graphs $Q_1=P_1[V_2,V_3']$ and $Q_2=P_2[V_1,V_3']$ are each $(\e/\d,\,d)$-regular.
	Since $f(x)/\d \le f^*(x)$, and since $Q_3=P_3$ is $(\e,d)$-regular by assumption, this would imply that $Q$ is $(f^*,d_2)$-regular, as needed.
	To prove that $Q_1$ is $(\e/\d,d)$-regular, let $S \sub V_2 \cup V_3'$ with $|\K(S)| \ge (\e/\d)|V_2||V_3'|$. Then $|\K(S)| \ge \e|V_2||V_3|$, hence $d_{Q_1}(S)=d_{P_1}(S) = d \pm \e$, as desired.
	Similarly, to prove that $Q_{2}$ is $(\e/\d,d)$-regular, let $S \sub V_1 \cup V_3'$ with $|\K(S)| \ge (\e/\d)|V_1||V_3'|$. Then $|\K(S)| \ge \e|V_1||V_3|$, hence $d_{Q_2}(S)=d_{P_2}(S) = d \pm \e$. This proves the induction basis.

	It remains to prove the induction step.
	Let $P$ be a $(k+1)$-complex on vertex classes $(V_1,\ldots,V_{k+1})$ and let $V_{k+1}' \sub V_{k+1}$ with $|V_{k+1}'| \ge \d|V_{k+1}|$, and suppose $P$ is $(f,d_2,\ldots,d_k)$-regular with
	\begin{equation}\label{eq:k-slice-f-bound}
	f(x) \le \frac{\d}{2} F_{k,\frac14}(x) \;.
	\end{equation}
	We need to show that the induced $(k+1)$-complex $Q=P[V_1,\ldots,V_k,V_{k+1}']$ is $(f^*,d_2,\ldots,d_k)$-regular.
	Put $d_0 = \min\{d_2,\ldots,d_{k-1}\}$,
	$P^{(k)}=(P_1,\ldots,P_{k+1})$ and $Q^{(k)}=(Q_1,\ldots,Q_{k+1})$.
	Let $i \in [k+1]$, and observe that the regularity assumption on $P$ translates to the following assumptions on $P_i$:
	\begin{enumerate}
		\item the $k$-complex $P_i^{(\le k-1)}$ is $(f,d_2,\ldots,d_{k-1})$-regular,
		\item the $k$-partite $k$-graph $P^{(k)}_i$ is $(f(d_0),\,d_k)$-regular in $P^{(k-1)}_i$.
	\end{enumerate}	
	To prove that $Q$ is $(f^*,d_2,\ldots,d_k)$-regular
	we need to show that $Q_i$ satisfies the following conditions:
	\begin{enumerate}
		\item the $k$-complex $Q_i^{(\le k-1)}$ is $(f^*,d_2,\ldots,d_{k-1})$-regular,
		\item the $k$-partite $k$-graph $Q^{(k)}_i$ is $(f^*(d_0),\,d_k)$-regular in $Q^{(k-1)}_i$.
	\end{enumerate}
	We henceforth assume $i \neq k+1$, since otherwise $Q_i=P_i$ and so the above conditions follow from the above assumptions together with the fact that $f(x) \le f^*(x)$.
	Apply the induction hypothesis with the $k$-complex $P_i^{(\le k-1)}$ and $V_{k+1}'$, using assumption~$(i)$, the fact that $f(x) \le \frac{\d}{2}F_{k-1,\frac14}(x)$ by~(\ref{eq:k-slice-f-bound}) and the statement's assumption on $|V(P)|$.
	It follows that the $k$-complex $Q_i^{(\le k-1)}=P_i^{(\le k-1)}[V_1,\ldots,V_{i-1},V_{i+1},\ldots,V_k,V_{k+1}']$ is $(f^*,d_2,\ldots,d_k)$-regular,
	thus proving condition~$(i)$.
	
	Apply Fact~\ref{fact:counting} (dense counting lemma) with $\g=1/2$ and the $k$-complex $P_i^{(\le k-1)}$, using assumption~$(i)$, the fact that $f(x) \le F_{k,\frac12}(x)$ by~(\ref{eq:k-slice-f-bound}) and the statement's assumption on $|V(P)|$, to deduce that
	$$|\K(P_i^{(k-1)})| \le \frac32\prod_{j=2}^{k-1} d_j^{\binom{k}{j}} \cdot \prod_{\substack{1 \le j \le k+1\colon\\j \neq i}} |V_j| \;.$$
	On the other hand, applying Fact~\ref{fact:counting} with $\g=1/4$ and the $k$-complex $Q_i^{(\le k-1)}$, using condition~$(i)$, the fact that $f^*(x) = \frac{2}{\d}f(x) \le F_{k,\frac14}(x)$ by~(\ref{eq:k-slice-f-bound})
	and the statement's assumption on $|V(P)|$, implies that
	\begin{equation}\label{eq:k-slice-count}
	|\K(Q_i^{(k-1)})| \ge \frac34\prod_{j=2}^{k-1} d_j^{\binom{k}{j}} \cdot \prod_{\substack{1 \le j \le k\colon\\j \neq i}} |V_j| \cdot \d |V_{k+1}| \ge \frac{\d}{2}|\K(P_i^{(k-1)})| \;.
	\end{equation}
	We now prove condition~$(ii)$.
	Let $S \sub Q_i^{(k-1)}$ satisfy $|\K(S)| \ge f^*(d_0)|\K(Q_i^{(k-1)})|$.
	Then $|\K(S)| \ge f(d_0)|\K(P_i^{(k-1)})|$ by~(\ref{eq:k-slice-count}).
	Therefore $d_{Q_i^{(k)}}(S)=d_{P_i^{(k)}}(S) = d_k \pm f(d_0)$, where the last equality uses assumption~$(ii)$.
	This proves condition~$(ii)$, thus completing the induction step and the proof.
\end{proof}

\subsection{Proof of Claim~\ref{claim:k-reduction}}

\begin{proof}
%
	Put $G=G_H^k$, $\d'=2\sqrt{\d}$, and let $E \in E_k(\P)$ and $V \in V_k(\P)$.
	Note that $E$ is a $(k-1)$-partite $(k-1)$-graph, and let $(V_1,\ldots,V_{k-1})$ denote its vertex classes. Thus, $V_j \sub \Vside^j$ for every $1 \le j \le k-1$, and $V \sub \Vside^k$.
	Moreover, let $E' \sub E$, $V' \sub V$ with $|E'| \ge \d'|E|$, $|V'| \ge \d'|V|$.
	%
	To complete the proof our goal is to show that
	$d_{G}(E',V') \ge \frac12 d_{G}(E,V)$ (recall Definition~\ref{def:star-regular}).

%
	
	Consider the following $k$-partite $k$-graph and subgraph thereof:
	$$K = \{ (v_1,\ldots,v_k) \,\vert\, (v_1,\ldots,v_{k-1}) \in E \text{ and } v_k \in V \} = E \circ V \;,$$
	$$K' = \{ (v_1,\ldots,v_k) \,\vert\, (v_1,\ldots,v_{k-1}) \in E' \text{ and } v_k \in V' \} = E' \circ V' \;.$$
	We claim that
	\begin{equation}\label{eq:red-k-d}
	d_{G}(E,V) = \frac{|H \cap K|}{|K|} \quad\text{ and }\quad d_{G}(E',V') = \frac{|H \cap K'|}{|K'|} \;.
	\end{equation}
	Proving~(\ref{eq:red-k-d}) would mean that to complete the proof it suffices to show that
	\begin{equation}\label{eq:red-k-goal}
	\frac{|H \cap K'|}{|K'|} \ge \frac12 \frac{|H \cap K|}{|K|}  \;.
	\end{equation}
	To prove~(\ref{eq:red-k-d}) first note that
	\begin{equation}\label{eq:red-k-prod}
	|K| = |E||V| \quad\text{ and }\quad |K'|=|E'||V'| \;.
	\end{equation}
	Furthermore,
	\begin{align*}
	e_G(E,V) &= \big|\big\{ \, ((v_1,\ldots,v_{k-1}),v_k) \in G \,\big\vert\, (v_1,\ldots,v_{k-1}) \in E,\, v_k \in V \, \big\}\big| \\
	&= \big|\big\{ \, (v_1,\ldots,v_{k-1},v_k) \in H \,\big\vert\, (v_1,\ldots,v_{k-1}) \in E,\, v_k \in V \, \big\}\big|
	= |H \cap K| \;,
	\end{align*}
	and similarly, $e_G(E',V')=|H \cap K'|$.
	Therefore, using~(\ref{eq:red-k-prod}), we indeed have
	$$d_G(E,V) = \frac{e_G(E,V)}{|E||V|}
	= \frac{|H \cap K|}{|K|}
	\quad\text{ and }\quad d_G(E',V') = \frac{e_G(E',V')}{|E'||V'|}
	= \frac{|H \cap K'|}{|K'|} \;.$$
	
	Having completed the proof of~(\ref{eq:red-k-d}),
	it remains to prove~(\ref{eq:red-k-goal}).
	By Claim~\ref{claim:decomposition} there is a set of $k$-polyads $\{P_i\}_i$ of $\P$ on $(V_1,\ldots,V_{k-1},V)$ such that
	\begin{equation}\label{eq:red-k-partitionP}
	K = \bigcup_i \K(P_i) \,\text{ is a partition, with } P_i = (P_{i,1},\ldots,P_{i,k-1},E) \;.
	\end{equation}
	For each $k$-polyad $P_i=(P_{i,1},\ldots,P_{i,k-1},E)$, let $P_i'$ be the induced $k$-polyad $P_i' = P_i[V_1,\ldots,V_{k-1},V']$. Write $P_i'=(P'_{i,1},\ldots,P'_{i,k},E)$,
	and let $Q_i$ be the $k$-polyad $Q_i=(P'_{i,1},\ldots,P'_{i,k-1},E')$.
	Note that $Q_i$ satisfies $\K(Q_i) = \K(P_i) \cap K'$.
	It therefore follows from~(\ref{eq:red-k-partitionP}) that
	\begin{equation}\label{eq:red-k-partitionQ}
	K' = \bigcup_i (\K(P_i) \cap K') = \bigcup_i \K(Q_i) \,\text{ is a partition.}
	\end{equation}
	%
	%
	%
	%
	Suppose $\P$ is a $(k-1,a_1,a_2,\ldots,a_{k-1})$-partition, and denote $d_j = 1/a_j$ and
	$$d = \prod_{j=2}^{k-1} d_j^{\binom{k-1}{j-1}} \;.$$
	Put $\g = \frac18\d'$ ($\le \frac18$, as otherwise there is nothing to prove). We  will next apply the dense counting lemma (Fact~\ref{fact:counting}) to prove the estimates:
	\begin{equation}\label{eq:red-k-tP}
	|\K(P_i)| \le (1+\g)d|K|
	\end{equation}
	and
	\begin{equation}\label{eq:red-k-tQ}
	|\K(Q_i)| \ge \big(1-\g)d|K'| \;.
	\end{equation}
	Note that proving these estimates would in particular imply the bound
	\begin{equation}\label{eq:red-k-threshold}
	|\K(Q_i)| \ge \d|\K(P_i)| \;;
	\end{equation}
	indeed, from the assumptions $|E'| \ge \d'|E|$, $|V'| \ge \d'|V|$ and~(\ref{eq:red-k-prod}) we have that $|K'| \ge (\d')^2|K|$,
	hence we deduce from~(\ref{eq:red-k-tP}) and~(\ref{eq:red-k-tQ}) the lower bound
	$$
	\frac{|\K(Q_i)|}{|\K(P_i)|}
	\ge \frac{1-\g}{1+\g}(\d')^2
	\ge \frac34 \cdot (2\sqrt{\d})^2 \ge \d \;,
	$$
	where we used the inequality
	\begin{equation}\label{eq:red-k-quotient}
	\frac{1-\g}{1+\g} \ge 1-2\g \ge \frac34 \;.
	\end{equation}
	
	In order to prove~(\ref{eq:red-k-tP}) and~(\ref{eq:red-k-tQ})
	we first need to introduce some notation.
	Put $m=n/a_1$ where $n=|V(H)|$ is the size of the vertex set, and put
	$$
	\g' = \frac12\g \d'd \,\,\Big(=\frac14\d d\Big), \quad d_0 = \min\{d_2,\ldots,d_{k-1}\} \;.$$	
	%
	Note that $d \ge d_0^{2^k}$.
	Using the statement's assumption on $f$ we have (recall~(\ref{eq:DCL}))
	\begin{equation}\label{eq:red-k-f}
	f(d_0) = \d^4\Big(\frac{d_0}{2}\Big)^{2^{k+3}}
	\le \frac{\d}{2^{6}} \cdot \Big(\frac14\d d_0^{2^k}\Big)^3 \Big(\frac{d_0}{2}\Big)^{2^{k+1}}
	\le \frac{\d}{2}\cdot\frac{\g'^3}{12}\Big(\frac{d_0}{2}\Big)^{2^{k+1}} = \frac{\d}{2}F_{k,\g'}(d_0) \;.
	\end{equation}
	In particular,
	\begin{equation}\label{eq:red-k-eps}
	f(d_0) \le \g' d_0 \le \g' d_{k-1} \;.
	\end{equation}
	Note that if $P$ is an $\ell$-polyad of $\P$, for any $2 \le \ell \le k$, then, since $\P$ is $f$-equitable, the unique $\ell$-complex of $\P$ containing $P$ is $(f,d_2,\ldots,d_{\ell-1})$-equitable.
	Applying Fact~\ref{fact:counting} (dense counting lemma) with $\g'$ implies, using the fact that $f(x) \le F_{k,\g'}(x) \le F_{\ell,\g'}(x)$ by~(\ref{eq:red-k-f}) and the statement's assumption on $n$, that
	\begin{equation}\label{eq:red-k-count}
	|\K(P)| = (1 \pm \g')\prod_{j=2}^{\ell-1} d_j^{\binom{\ell}{j}} \cdot m^\ell \;.
	\end{equation}
	%

	%
	Let $P_E$ be the unique $(k-1)$-polyad of $\P$ such that $E \sub K(P_E)$.
	Then $|E|=d_E(P_E)|\K(P_E)|$, and since $\P$ is $f$-equitable, $E$ is $(d_{k-1},f(d_0))$-regular in $P_E$. In particular, $|E| \ge (d_{k-1}-f(d_0))|\K(P_E)|$.
	By~(\ref{eq:red-k-count}),
	$$|\K(P_E)| \ge (1-\g')\prod_{j=2}^{k-2} d_j^{\binom{k-1}{j}} \cdot m^{k-1} \;.$$
	Thus,
	\begin{equation}\label{eq:red-k-E}
	|E| \ge (d_{k-1}-f(d_0))|\K(P_E)| \ge (1-\g')d_{k-1}|\K(P_E)| \ge 
	(1-2\g')\prod_{j=2}^{k-1} d_j^{\binom{k-1}{j}} \cdot m^{k-1} \;,\
	\end{equation}
	where the second inequality uses~(\ref{eq:red-k-eps}).
	Furthermore, for every $P_i$ as above we have (recall $P_i$ is a $k$-polyad of $\P$), again by~(\ref{eq:red-k-count}), that
	%
	$$
	|\K(P_i)| \le (1+\g')\prod_{j=2}^{k-1} d_j^{\binom{k}{j}} \cdot m^k
	= (1+\g')d\prod_{j=2}^{k-1} d_j^{\binom{k-1}{j}} \cdot m^k
	\le \frac{1+\g'}{1-2\g'}d|E||V|
	\le (1+4\g')d|K| \;,
	$$
	where the penultimate inequality uses~(\ref{eq:red-k-E}), and the last inequality uses~(\ref{eq:red-k-prod}) and the fact that $\g' \le \g \le \frac18$.
	%
	This proves~(\ref{eq:red-k-tP}).
	
	
	Next we prove~(\ref{eq:red-k-tQ}). Let $\overline{P_i}$ be the unique $k$-complex of $\P$ containing the $k$-polyad $P_i$, and let $\overline{P_i}'$ be the induced $k$-complex $\overline{P_i}'=\overline{P_i}[V_1,\ldots,V_{k-1},V']$.
	Apply Lemma~\ref{lemma:k-slice} (slicing lemma) on $\overline{P_i}$, using the fact that $|V'| \ge \d'|V|$ and $f(x) \le \frac{\d}{2}F_{k-1,\frac14}$ by~(\ref{eq:red-k-f}), to deduce that $\overline{P_i}'$ is $(\frac{2}{\d}f,\, d_2,\ldots,d_{k-1})$-regular.
	Let the $k$-complex $\overline{Q_i}$ be obtained from the $k$-complex $\overline{P_i}'$ by replacing $E$ $(=P[V_1,\ldots,V_{k-1}])$ with $E'$,
	and note that the $(k-1)$-uniform hypergraph $\overline{Q_i}^{(k-1)}$ is precisely the $k$-polyad $Q_i$.
	%
	Apply Fact~\ref{fact:counting} (dense counting lemma, the ``moreover'' part) with $\g'$ on $\overline{P_i}'$, using the fact that 
	$\frac{2}{\d}f(x) 
	\le F_{k,\g'}(x)$ by~(\ref{eq:red-k-f}) and the statement's assumption on $n$, to deduce that
	$$
	|\K(Q_i)| \ge |E'| \cdot (1-\g')d|V'| - \g'|E|\cdot|V'|
	\ge \big(1-\g' - \frac{1}{\d' d}\g'\big)d|E'||V'|
	\ge \big(1-\g)d|K'| \;,
	$$
	where the second inequality uses the assumption that $|E'|\ge \d'|E|$ and the third inequality uses~(\ref{eq:red-k-prod}).
	This proves~(\ref{eq:red-k-tQ}).

	Finally, recall that our goal is to prove~(\ref{eq:red-k-goal}).
	We have
	\begin{align*}
	|H \cap K| &= \sum_i |H \cap \K(P_i)|
	= \sum_i d_H(P_i) \cdot |\K(P_i)|
	\le (1+\g)|K| \cdot d\sum_i d_H(P_i) \;,
	\end{align*}
	where the first equality uses~(\ref{eq:red-k-partitionP}) and the inequality is by~(\ref{eq:red-k-tP}).
	Put $d'=d\sum_i d_H(P_i)$. Then
	\begin{equation}\label{eq:red-k-dP}
	\frac{|H \cap K|}{|K|} \le (1 + \g)d' \;.
	\end{equation}
	Observe that for every $i$, the statement's assumption on $\P$
	implies, together with~(\ref{eq:red-k-threshold}), that
	\begin{equation}\label{eq:red-k-reg}
	d_H(Q_i) \ge \frac23 d_H(P_i) \;.
	\end{equation}
	We have
	\begin{align*}
	|H \cap K'| &= \sum_i |H \cap \K(Q_i)| =
	\sum_i d_H(Q_i) \cdot |\K(Q_i)| \\
	&\ge \sum_i \frac23 d_H(P_i) \cdot |\K(Q_i)|
	\ge \frac23(1-\g)|K'| \cdot d\sum_i d_H(P_i) \;,
	\end{align*}
	where the first equality uses~(\ref{eq:red-k-partitionQ}), the first inequality uses~(\ref{eq:red-k-reg}) and the second inequality uses~(\ref{eq:red-k-tQ}).
	This means that
	$$\frac{|H \cap K'|}{|K'|} \ge \frac23(1-\g)d' \ge
	\frac23\cdot\frac{1-\g}{1+\g}\frac{|H \cap K|}{|K|} \ge \frac12 \frac{|H \cap K|}{|K|} \;,$$
	where the second inequality uses~(\ref{eq:red-k-dP}) and the third inequality uses~(\ref{eq:red-k-quotient}).
	We have thus proved~(\ref{eq:red-k-goal}) and are therefore done.
\end{proof}

\end{document}